\documentclass[11pt]{article}
\usepackage{amsthm, amsmath, amssymb, amsfonts, url, booktabs, tikz, setspace, fancyhdr, bm}
\usepackage{cancel}
\usepackage{geometry}
\usepackage{hyperref, enumerate}
\usepackage[shortlabels]{enumitem}
\usepackage[babel]{microtype}
\usepackage[english]{babel}
\usepackage[capitalise]{cleveref}
\usepackage{comment}
\usepackage{bbm}
\usepackage{csquotes}
\usepackage{mathabx}
\usepackage{tikz}
\usepackage{subcaption}
\usepackage{graphicx}
\usepackage{float}
\usepackage{xcolor}
\usepackage[normalem]{ulem}
\usepackage{diagbox}
\usetikzlibrary{positioning, arrows.meta, shapes.geometric}

\counterwithin{figure}{section}

\newtheorem{theorem}{Theorem}[section]
\newtheorem{prop}[theorem]{Proposition}
\newtheorem{lemma}[theorem]{Lemma}
\newtheorem{cor}[theorem]{Corollary}

\newtheorem{claim}[theorem]{Claim}

\newtheorem{fact}[theorem]{Fact}

\theoremstyle{definition}

\newtheorem{definition}[theorem]{Definition}
\newtheorem*{defn-non}{Definition}

\newtheorem{ques}[theorem]{Question}
\definecolor{rosepink}{RGB}{255,102,204}
\definecolor{dateplum}{HTML}{993366}
\definecolor{darkdateplum}{RGB}{128,0,32}
\definecolor{lightdateplum}{RGB}{219,112,147}
\definecolor{darkred}{RGB}{139,0,0}
\definecolor{lightred}{RGB}{240,130,100}

\newlist{Case}{enumerate}{3}
\setlist[Case, 1]{%
    label           =   {\bfseries Case \arabic*.},
    labelindent=1em ,labelwidth=1cm, labelsep*=1em, leftmargin =!
}
\setlist[Case, 2]{%
    label           =   {\bfseries Subcase \arabic{Casei}.\arabic*.},
    labelindent=-1em ,labelwidth=1cm, labelsep*=1em, leftmargin =!
}
\setlist[Case, 3]{%
    label           =   {\bfseries Subsubcase \arabic{Casei}.\arabic{Caseii}.\arabic*.},
    labelindent=-1em ,labelwidth=1cm, labelsep*=1em, leftmargin =!
}

\newenvironment{poc}{\begin{proof}[Proof of claim]}{\end{proof}}

\usepackage{todonotes}

\newcommand{\eps}{\varepsilon}

\title{On the spectrum and structure of blowup thresholds}

\author{
Xinqi Huang\thanks{School of Mathematical Sciences, University of Science and Technology of China, Hefei, China and Extremal Combinatorics and Probability Group (ECOPRO), Institute for Basic Science (IBS), Daejeon, South Korea.
Email: \texttt{huangxq@mail.ustc.edu.cn}. Supported by the USTC Excellent PhD Students Overseas, the Institute for Basic Science (IBS-R029-C4), the National Key Research and Development Programs of China 2023YFA1010200, the NSFC under Grants No. 12171452 and No. 12231014 and Innovation Program for Quantum Science and Technology 2021ZD0302902.
}
\and
Hong Liu\thanks{Extremal Combinatorics and Probability Group (ECOPRO), Institute for Basic Science (IBS), Daejeon, South Korea. Emails: \texttt{hongliu@ibs.re.kr}. Supported by the Institute for Basic Science (IBS-R029-C4).}
\and
Mingyuan Rong\thanks{School of Mathematical Sciences, University of Science and Technology of China, Hefei,
China. Email: \texttt{rong\_ming\_yuan@mail.ustc.edu.cn}. Research supported by National Key Research and Development Program of China 2023YFA1010201, the NSFC under Grant No. 12125106 and the Excellent PhD Students Overseas Study Program of the University of Science and Technology of China.}
}

\begin{document}
\date{}
\maketitle

\begin{abstract}
The chromatic threshold of Erd\H{o}s and Simonovits asks when a
minimum-degree condition forces every \(H\)-free graph to have bounded
chromatic number.  Thomassen's homomorphism threshold strengthens this by
requiring a bounded \(H\)-free homomorphic image.  The recently introduced
blowup threshold \(\delta_{\mathrm B}(H)\) asks for a still more rigid
conclusion: when must every sufficiently dense maximal \(H\)-free graph be an
actual blowup of a bounded graph?  Thus the blowup threshold measures when
quotient-level structure can be upgraded to exact bounded-template structure.

We show that, although chromatic and homomorphism thresholds are often hard to
separate, the stronger blowup threshold diverges from the chromatic threshold in
several fundamental ways.

First, we prove that
\(\delta_{\mathrm B}(H)>0\) for every non-bipartite graph \(H\).  Hence, unlike
the chromatic threshold, the blowup threshold never vanishes outside the
bipartite world.

Second, we prove that $\delta_{\mathrm B}$ is not monotone under taking
induced subgraphs. This shows that the blowup threshold is sensitive to
global features of the forbidden graph and cannot be classified by a
direct analogue of the monotonicity-based strategy used for chromatic
thresholds.

Third, we prove that $\delta_{\mathrm B}(H)=\frac{1}{4}$ for a natural family of
\(3\)-chromatic constrained blowups of odd cycles. This gives a new exact
blowup-threshold value beyond the chromatic-threshold spectrum.
\end{abstract}

\section{Introduction}
\subsection{Background and related work}

A recurring theme in extremal graph theory is that dense graphs avoiding a fixed
subgraph often have bounded structural complexity.  The form of this bounded
complexity, however, can vary substantially.  At the coarsest level, one may ask
only for bounded chromatic number.  A stronger requirement is the existence of a
bounded forbidden-subgraph-free quotient.  Stronger still is the demand that the
graph itself be an exact blowup of a bounded template.  The purpose of this paper
is to study the minimum-degree threshold at which this last, exact form of
template rigidity is forced.

The modern theory begins with the chromatic threshold problem, proposed by
Erd\H{o}s and Simonovits~\cite{1973ErdosSimonovits}.  The \emph{chromatic
threshold} of a graph \(H\) is
\[
\delta_\chi(H):=
\inf\Bigl\{\alpha\ge 0:\exists\, C=C(H,\alpha)
\textup{ s.t. \(\forall\ H\)-free \(G\), if }
\delta(G)\ge \alpha |V(G)|,
\textup{ then } \chi(G)\le C\Bigr\}.
\]
The first central case is \(H=K_3\): Hajnal's Kneser-graph construction
gives \(\delta_\chi(K_3)\ge 1/3\)~\cite{1973ErdosSimonovits}, and Thomassen proved
the matching upper bound~\cite{2002Thomassen}.  The clique case was later settled
by Goddard--Lyle~\cite{2011JGTKrChromatic} and independently by
Nikiforov~\cite{2010arxivKrfree}, who proved $\delta_\chi(K_s)=\frac{2s-5}{2s-3}.$
In contrast, Thomassen showed that
$\delta_\chi(C_{2k-1})=0$
 for every $k\ge 3$~\cite{2007OddCycleChromatic}. After substantial further work
\cite{bottcher2023graphs,2011Unpubilished,1982OddCycles,1995DMJin,2010ColoringViaVCDim,ning2026chromatic},
Allen, B\"ottcher, Griffiths, Kohayakawa, and Morris~\cite{2013advAllChromatic}
gave the complete classification: if \(\chi(H)=r\ge 3\), then
    $\delta_\chi(H)\in
    \left\{
        \frac{r-2}{r-1},\
        \frac{2r-5}{2r-3},\
        \frac{r-3}{r-2}
    \right\}.$
Thus the chromatic-threshold spectrum is remarkably rigid.

Bounded chromatic number is equivalent to admitting a homomorphism to a bounded
complete graph.  Motivated by this viewpoint, Thomassen~\cite{2002Thomassen}
proposed a natural strengthening: require the bounded homomorphic image itself to
be \(H\)-free.  Recall that \(G\xrightarrow{\textup{hom}}F\) means that there is
an adjacency-preserving map from \(V(G)\) to \(V(F)\).  This leads to the
\emph{homomorphism threshold}
\[
\delta_{\textup{hom}}(H):=
\inf\Bigl\{\alpha\ge 0:\exists\ H\textup{-free }F=F(H,\alpha)\
\textup{ s.t. \(\forall\ H\)-free \(G\), if }
\delta(G)\ge \alpha |V(G)|,
\textup{ then } G\xrightarrow{\textup{hom}}F\Bigr\}.
\]
Clearly,
$    \delta_{\textup{hom}}(H)\ge \delta_{\chi}(H).$

The first major results showed that equality holds for cliques.
{\L}uczak~\cite{2006CombTriangleHom} determined the triangle case, and
Goddard--Lyle~\cite{2011JGTKrChromatic} proved for every \(s\ge 3\) that    $\delta_{\textup{hom}}(K_s)
    =
    \delta_{\chi}(K_s)
    =
    \frac{2s-5}{2s-3}.$
The known proofs use Szemer\'edi's regularity lemma~\cite{1978OriginalRegularity};
subsequent work obtained quantitative refinements and alternative proofs
\cite{2020CPCProb,2024GraphToGeom}. The clique case raises a basic question: can the inequality
\(\delta_{\textup{hom}}(H)\ge \delta_{\chi}(H)\) be strict?  Odd cycles provide
the first natural testing ground.  Thomassen proved
\(\delta_{\chi}(C_{2k-1})=0\) for every \(k\ge 3\), while
Letzter--Snyder~\cite{2019JGTC3C5} and Ebsen--Schacht~\cite{2020COMBHomoOddCycle}
proved   $ \delta_{\textup{hom}}(C_{2k-1})\le \frac{1}{2k-1}.$
Very recently, Sankar~\cite{2022Maya} used topological methods to obtain the
first, and so far only, separation between chromatic and homomorphism thresholds
for single forbidden graphs:   $ \delta_{\textup{hom}}(C_{2k-1})>0$ for all  $k\ge 3$.
Thus, beyond cliques, even separating the first two levels of structure is
subtle.

A homomorphism to a bounded graph still gives only a quotient-level description.
It says that \(G\) is a subgraph of a blowup of the image, but the pairs of
fibres corresponding to template edges may have been arbitrarily thinned.  If
\(G\) is maximal \(H\)-free, it is natural to ask whether this loss of information
can be eliminated: must the relevant fibre pairs actually be complete?  Equivalently,
when does high minimum degree force a dense maximal \(H\)-free graph to have a
bounded twin quotient?

This is the motivation for the \emph{blowup threshold}, recently introduced
in~\cite{huang2025interpolatingchromatichomomorphismthresholds}.  For a graph
\(F\), a blowup of \(F\) is obtained by replacing every vertex of \(F\) by an
independent set and every edge of \(F\) by a complete bipartite graph between the
corresponding parts.  We write \(F[\cdot]\) for an arbitrary blowup of \(F\)
\footnote{Here the independent set replacing a vertex is allowed to be empty.}.
The blowup threshold of \(H\) is
\[
\delta_{\textup{B}}(H):=
\inf\big\{\alpha\ge 0:\exists~F=F(\alpha,H)\textup{ s.t. \(\forall\) maximal \(H\)-free \(G\), if }
\delta(G)\ge \alpha |V(G)|,
\textup{ then } G=F[\cdot]\big\}.
\]

Since every blowup of a bounded graph has a bounded homomorphic image, we have that
$    \delta_{\textup{B}}(H)\ge \delta_{\textup{hom}}(H)\ge \delta_{\chi}(H).$ Thus the blowup threshold asks for a stricter form of bounded structure: not only
a bounded quotient, but an exact bounded template.  This stricter requirement
creates a useful asymmetry.  Lower bounds can often be proved by constructing
dense maximal \(H\)-free graphs with unbounded twin quotient, which is more
concrete than showing that no bounded \(H\)-free homomorphic image exists.  Upper
bounds, on the other hand, are harder: one must upgrade an approximate or
quotient-level structure into a genuinely homogeneous bounded partition.

This upgrading problem gives a second, closely related motivation.  In regularity
arguments for dense \(H\)-free graphs, and already in {\L}uczak's proof for the
triangle case, maximality often leads to a bounded partition in which most pairs
of parts have density close to either \(0\) or \(1\).  It is therefore natural to
ask for conditions under which such an almost homogeneous partition can be boosted
to a genuine homogeneous one, equivalently to an actual blowup of a bounded graph.
This question also connects naturally with bounded-VC regularity theory: for
graphs of bounded VC-dimension, one can often obtain bounded almost homogeneous
partitions \cite{AlonFischerNewman,LovaszSzegedy}.  From this viewpoint, the
blowup threshold measures when approximate template structure is forced to become
exact.

The first known examples already show that this strengthening is nontrivial.  For
cliques, the chromatic, homomorphism, and blowup thresholds coincide
\cite{2024GraphToGeom}.  For odd cycles, however, the blowup threshold can be
determined exactly~\cite{huang2025interpolatingchromatichomomorphismthresholds}:    $\delta_{\textup{B}}(C_{2k-1})=\frac{1}{2k-1}.$
This sharply contrasts with the chromatic threshold, since
\(\delta_{\chi}(C_{2k-1})=0\) for \(k\ge 3\).  Thus the blowup threshold gives an
exact answer for the most classical family with vanishing chromatic
threshold, where the corresponding homomorphism threshold remains much
harder to pin down exactly.

\subsection{Our contributions}

We study the extent to which the blowup threshold behaves like, and unlike, the
chromatic and homomorphism thresholds.  It is useful to organize the discussion in
terms of spectra.  Write
\[
    \Delta_\star
    :=
    \{\delta_\star(H):\chi(H)\ge 3\},
    \qquad
    \star\in\{\chi,\textup{hom},\textup{B}\},
\]
and, for \(r\ge 3\), $\Delta_\star^{(r)}
    :=
    \{\delta_\star(H):\chi(H)=r\}.$
For example, the chromatic-threshold classification gives    $\Delta_\chi^{(3)}
    =
    \left\{0,\frac{1}{3},\frac{1}{2}\right\}.$
Our results show that the blowup-threshold spectrum is substantially richer and
that the parameter is structurally more delicate.

\paragraph{Positivity.}
Our first result shows that the blowup threshold never vanishes on non-bipartite
graphs.

\begin{theorem}\label{thm: general lowbdd for 3-color blowup}
For every graph \(H\) with \(\chi(H)\ge 3\), we have
   $ \delta_{\textup{B}}(H)>0.$
Equivalently, \(0\notin \Delta_{\textup{B}}\).
\end{theorem}

This is in sharp contrast with the chromatic threshold, which vanishes for many
non-bipartite graphs, including all odd cycles \(C_{2k-1}\) with \(k\ge 3\).
For \(\chi(H)\ge 4\), positivity follows immediately from $\delta_{\textup{B}}(H)\ge \delta_\chi(H)>0.$
The new content is the \(3\)-chromatic case, where the chromatic threshold may be
zero.  We prove positivity with an explicit pseudo-blowup construction whose twin
quotient is necessarily unbounded.

\paragraph{Non-monotonicity.}
Our second result shows that the difference from chromatic thresholds is not only
spectral but also structural.  The chromatic threshold is monotone under taking
subgraphs: if \(H'\subseteq H\), then every \(H'\)-free graph is \(H\)-free, and
hence    $\delta_\chi(H')\le \delta_\chi(H).$
This monotonicity is one of the key simplifications behind the classification of
chromatic thresholds, since it allows one to reduce to canonical minimal
obstructions.  The blowup threshold does not have this property, even under taking induced subgraphs.

\begin{theorem}\label{thm: non-monotone}
The blowup threshold is not monotone under taking induced subgraphs.
\end{theorem}

More precisely, we construct a graph \(H\) such that
    $\delta_{\textup{B}}(H)
    >
    \delta_{\textup{B}}(H\sqcup K_3),$
although \(H\) is an induced subgraph of \(H\sqcup K_3\).  This shows that
\(\delta_{\textup{B}}\) is sensitive to global features of the forbidden graph,
not merely to smaller induced obstructions.  Consequently, a classification of
blowup thresholds cannot be obtained by directly imitating the monotonicity-based
strategy used for chromatic thresholds.

\paragraph{A new exact value.}
Our third result gives an exact blowup-threshold value outside the spectrum of
the chromatic threshold $\delta_\chi$. More precisely, the proof determines
the value $\frac{1}{4}$ for a natural family of $3$-chromatic graphs obtained
from constrained blowups of odd cycles.

\begin{theorem}\label{thm:existence of 1/2k}
There exists a graph $H$ with $\chi(H)=3$ such that
$    \delta_{\textup{B}}(H)=\frac{1}{4}.$
\end{theorem}

By the chromatic-threshold classification, \(1/4\) is not a possible value of
\(\delta_\chi(H)\).  Thus this theorem gives a concrete value which is new from
the viewpoint of chromatic thresholds.  It also provides a useful benchmark for
the less-understood spectra of the homomorphism and VC thresholds
\cite{huang2025interpolatingchromatichomomorphismthresholds}.

\subsection{Overview of the proofs}

\noindent\textbf{Pseudo-blowups and lower bounds.}
The common lower-bound mechanism is to construct dense maximal \(H\)-free graphs
whose twin quotients are necessarily large.  We start from a fixed model graph
\(M\), replace its vertices by blocks, and keep most adjacent pairs complete
bipartite.  On selected pairs of small blocks, however, we replace the complete
bipartite graph by a sparse ordered pattern.
We call the resulting graph a pseudo-blowup.

The key point is that sparse pairs prevent representation as a blowup of any
bounded template.  Indeed, if a sufficiently large pseudo-blowup of \(M\) were
contained in a blowup of a bounded graph, then a pigeonhole argument would force a
genuine blowup \(M[t]\) to appear.  Thus, whenever \(H\subseteq M[t]\),
an \(H\)-free pseudo-blowup of \(M\) yields a lower bound for
\(\delta_{\textup{B}}(H)\).  This framework is used both for the positivity
theorem and for the general lower-bound theorem for the constrained odd-cycle
blowups.

\medskip
\noindent\textbf{Non-monotonicity.}
For the non-monotonicity result, we construct a graph \(H\) with    $\delta_{\textup{B}}(H)\ge \frac{1}{2}$
using another pseudo-blowup construction.  This construction exploits a sparse
matching-like pair that prevents bounded twin quotient while keeping the graph
\(H\)-free.

The behavior changes after adding a disjoint triangle.  The maximal
\((H\sqcup K_3)\)-free condition becomes much more restrictive: in any graph of
minimum degree slightly below \(n/2\), all non-bipartite behavior is forced into a
bounded exceptional set.  Once this set is identified, maximality allows us to add
all missing edges of a suitable complete bipartite extension without creating
\(H\sqcup K_3\).  The resulting graph is a blowup of a bounded template, and hence
so is the original graph.  This gives    $\delta_{\textup{B}}(H\sqcup K_3)<\frac{1}{2}\le \delta_{\textup{B}}(H),$
proving non-monotonicity under induced subgraphs.

\medskip
\noindent\textbf{The value \(\frac{1}{4}\).}
The graphs realizing \(\frac{1}{4}\) are constrained blowups of odd cycles.  For integers \(s\ge 3\) and \(0\le \ell\le s\), let
\(\mathfrak C_s^{\ge \ell}\) denote the family of blowups of \(C_s\) with a
singleton interval of length at least \(\ell\), that is, with at least \(\ell\)
consecutive parts of size \(1\).  Let
\(\mathfrak C_s^{\ell}\subseteq \mathfrak C_s^{\ge \ell}\) denote the subfamily
with a singleton interval of length exactly \(\ell\), and with all remaining parts
of size at least \(2\) (see
\Cref{fig: special blowup of C7}).

The exact value \(1/4\) follows from an upper bound below and a matching lower bound.

\begin{theorem}\label{thm: uppbdd for 1/2k}
Let \(s\ge 4\).  If  $ H\in \mathfrak C_{2s-1}^{\ge 4},$
then
$    \delta_{\textup{B}}(H)\le \frac{1}{4}.$
\end{theorem}

\begin{theorem}\label{thm: lowbdd for 1/2k}
Let \(s\ge 2k\ge 4\).  If
$    H\in \mathfrak C_{2s-1}^{4k-4},$
then $\delta_{\textup{B}}(H)\ge \frac{1}{2k}.$
\end{theorem}

Taking \(k=2\), we obtain
$\delta_{\textup{B}}(H)=\frac14$
for every \(H\in \mathfrak C_{2s-1}^{4}\). The lower bound is obtained from a pseudo-blowup of an odd cycle with one pendant
edge.  The sparse pairs are arranged so that any copy of
\(H\in \mathfrak C_{2s-1}^{4k-4}\) would force too long a singleton interval in
the model cycle, contradicting the construction.

The upper bound is the technical part.  Let \(G\) be a maximal \(H\)-free graph
with    $\delta(G)>\left(\frac{1}{4}+\varepsilon\right)|G|.$
A stability theorem for graphs forbidding odd-cycle blowups first shows that
\(G\) is almost bipartite.  The difficulty is to promote this approximate
bipartite structure to an exact one.  In the case needed here, any odd cycle can
be thickened, using the minimum-degree condition, into a constrained odd-cycle
blowup containing \(H\).  Hence \(G\) is bipartite, and maximality then forces it
to be complete bipartite, a blowup of a single edge.

\subsection{Organization of the paper}

The rest of the paper is organized as follows.  In \Cref{sec:lowbound}, we develop
the pseudo-blowup framework and use it to prove the positivity theorem and the
lower bound for the constrained odd-cycle blowups.  In \Cref{sec:non-monotone},
we prove the non-monotonicity theorem.  In \Cref{sec:proof-half-k}, we prove the
corresponding upper bound for the constrained odd-cycle blowups, which together
with the lower bound gives the value \(\frac{1}{4}\).  Finally, in
\Cref{sec:concluding-remarks}, we discuss open problems concerning the spectrum
and structure of blowup thresholds.

\section{Lower bound constructions}\label{sec:lowbound}

We develop a general pseudo-blowup reduction and apply it to the lower bounds in \Cref{thm: lowbdd for 1/2k,thm: general lowbdd for 3-color blowup}.  We first recall the balanced blowup notation. Let $M$ be a graph with vertex set $V(M)$. For a positive integer $t$, the \emph{$t$-blowup} of $M$, denoted by $M[t]$, is the graph obtained by replacing each vertex $v \in V(M)$ with an independent set $B_v$ of size $t$. For every edge $uv \in E(M)$, the pair $(B_u, B_v)$ induces a complete bipartite graph.

However, the complete bipartite connections in a standard blowup are often too restrictive for our purposes. To obtain our optimal bounds, we require a more flexible structure that allows for sparse interactions. We introduce a variation of the blowup where certain complete bipartite connections are replaced by specific sparse structures.

\begin{definition}
    Let $M$ be a graph and let $V(M) = V_1 \cup V_2$ be a partition of its vertex set. Given positive integers $n_1>n_2$, the \emph{$(n_1, n_2)$-pseudo-blowup} of $M$ with respect to this partition is the graph $\overline{M}$ constructed as follows.
    \begin{itemize}
        \item \textbf{Vertices:} For each vertex $v \in V_1$, replace it with an independent set $B_v$ of size $n_1$; we refer to $B_v$ as a \emph{large block}. For each vertex $v \in V_2$, replace it with an independent set $B_v$ of size $n_2$; we refer to it as a \emph{small block}. The vertices in each small block are indexed by $\{1, \dots, n_2\}$.
        \item \textbf{Edges:} For every edge $uv \in E(M)$, we place edges between the corresponding blocks $B_u$ and $B_v$. If at least one of $B_u$ or $B_v$ is a large block, the edges form a complete bipartite graph.
        If both $B_u$ and $B_v$ are small blocks, the connection between them is governed by a rule specific to the pair $uv$. Under this rule, a vertex with index $i$ in $B_u$ and a vertex with index $j$ in $B_v$ are adjacent if and only if their indices satisfy a designated condition. This designated condition must be exactly one of the following: $i=j$, $i<j$, or $i>j$.
        No other edges are present.
    \end{itemize}
\end{definition}

The following lemma connects the pseudo-blowup to the standard blowup. We establish that if a pseudo-blowup is embedded within a blowup of a graph of bounded size, it enforces the existence of a standard blowup. 

\begin{lemma}\label{lemma: lowerbound}
Let $C$ and $t\geq 2$ be positive integers. Let $M$ be a graph. Let $\overline{M}$ be a $(n_1, n_2)$-pseudo-blowup of $M$ defined by a partition $V_1 \cup V_2$ such that the parameters satisfy $n_1 >n_2 > t \cdot C^{|V(M)|}$. If a graph $G$ contains $\overline{M}$ as a subgraph and is also a blowup of some graph $F$ with $|V(F)| \le C$, then $G$ must contain $M[t]$ as a subgraph.
\end{lemma}

\begin{proof}

Since $G$ is a blowup of $F$, fix a blowup partition of $G$ over $F$ and let $\phi\colon V(G)\to V(F)$ be its canonical projection. For each $v \in V(M)$, let $B_v$ denote the corresponding block in $\overline{M}$. We first select the subsets from the large blocks. For each $v \in V_1$, we simply choose an arbitrary subset $A_v \subseteq B_v$ of size $t$.

Next, we select the subsets from the small blocks. For each index $j\in[n_2]$, we consider the sequence of images under $\phi$ of the $j$-th vertices across all small blocks. The number of possible such sequences is $|V(F)|^{|V_2|} \le C^{|V(M)|}$. Since $n_2 > t \cdot C^{|V(M)|}$, the pigeonhole principle guarantees the existence of a set of indices $J$ of size $t$ where these sequences are identical. For each $v \in V_2$, we define a subset $A_v \subseteq B_v$ of size $t$ consisting precisely of the vertices with indices in $J$. By our choice of $J$, all vertices in $A_v$ map to the same vertex in $F$.

It remains to show that the selected sets $\{A_v\}_{v \in V(M)}$ form a copy of $M[t]$ in $G$. Let $xy$ be an edge in $M$. If at least one of $x$ or $y$ is in $V_1$, the blocks $B_x$ and $B_y$ are fully connected in $\overline{M}$, which implies that $A_x$ and $A_y$ are fully connected in $G$. If both $x$ and $y$ are in $V_2$, we use the fact that $|J| = t \ge 2$. Depending on the designated rule for the pair $xy$, we can always find indices $i, j \in J$ such that the vertex with index $i$ in $A_x$ and the vertex with index $j$ in $A_y$ are adjacent in $\overline{M}$. Specifically, if the rule is $i=j$, we select the same index from $J$. If the rule is $i<j$ or $i>j$, we select two distinct indices from $J$ to satisfy the required inequality. Recall that by construction, both $A_x$ and $A_y$ map to single vertices in $F$. The existence of this edge between $A_x$ and $A_y$ implies that their corresponding target vertices in $F$ are adjacent. Since $G$ is a blowup of $F$, this adjacency guarantees that $A_x$ and $A_y$ induce a complete bipartite graph in $G$.
\end{proof}

The lemma gives the following reduction.

\begin{cor}\label{cor: main reduction}
Let $\alpha \in [0,1]$, $t\geq 2$ be a positive integer, and $H, M$ be graphs such that $H \subseteq M[t]$. If for every sufficiently large constant $C$, there exists an $H$-free $(n_1, n_2)$-pseudo-blowup $\overline{M}$ of $M$ with parameters satisfying $n_1 > n_2 > t \cdot C^{|V(M)|}$ and $\delta(\overline{M}) \ge (\alpha-\frac{1}{C}) |V(\overline{M})|$, then $\delta_{\textup{B}}(H) \ge \alpha$.
\end{cor}

\begin{proof}[Proof of Corollary~\ref{cor: main reduction}]
Suppose to the contrary that $\delta_{\textup{B}}(H)<\alpha$.  Choose
$\beta$ with $\delta_{\textup{B}}(H)<\beta<\alpha$ and a graph $F$ witnessing
$\beta$ in the definition of the blowup threshold.  Thus every maximal
$H$-free graph $G$ with $\delta(G)\ge \beta |V(G)|$ is a blowup of $F$.

Choose $C$ sufficiently large that the hypothesis applies, $C\ge |V(F)|$, and
$1/C\le \alpha-\beta$.  Let $\overline M$ be the resulting $H$-free
pseudo-blowup, and extend it on the same vertex set to a maximal $H$-free graph
$G$.  Then
\[
    \delta(G)\ge \delta(\overline M)
    \ge \left(\alpha-\frac1C\right)|V(G)|
    \ge \beta |V(G)|.
\]

Consequently, $G$ must be a blowup of $F$. Since $G$ contains $\overline{M}$ and is a blowup of a graph of size $|V(F)| \le C$, Lemma \ref{lemma: lowerbound} implies that $G$ contains $M[t]$ as a subgraph. Since $H \subseteq M[t]$, this means $G$ contains $H$, contradicting that $G$ is $H$-free.
\end{proof}

We now construct the model graphs $M$ required for the two applications.

\subsection{Proof of \Cref{thm: lowbdd for 1/2k}}
Throughout this subsection, every edge between two small blocks is defined by the rule $i=j$, and hence forms a perfect matching.  We use the following model graph.

\begin{definition}\label{def: model M}
    Let $k \ge 2$ be an integer. The graph $M$ consists of a cycle of length $4k-1$ with a pendant edge attached to one of its vertices. Formally, let $V(M) = \{v_1, \dots, v_{4k}\}$. The edge set $E(M)$ comprises the edges of the cycle $v_1 v_2 \dots v_{4k-1} v_1$ together with the pendant edge $v_{4k-1}v_{4k}$.
    
    We define a partition $V(M) = V_1 \cup V_2$ based on the indices modulo 4 as follows:
    $$
        V_1 = \{v_i \in V(M) : i \equiv 0, 3 \pmod 4\} \quad \text{and} \quad V_2 = \{v_i \in V(M) : i \equiv 1, 2 \pmod 4\}.
    $$
    Note that this partition satisfies $|V_1| = |V_2| = 2k$.
\end{definition}

The figure below illustrates the construction for $k=2$.
On the left is the model graph $M$: blue vertices belong to $V_1$, red vertices belong to $V_2$.
On the right is the pseudo-blowup $\overline{M}$ with rule $i=j$: large circles are large blocks (size $n_1$), small circles are small blocks (size $n_2$), solid edges are complete bipartite graphs, and dashed edges are perfect matchings.
Notice that every vertex is adjacent to some large block, so $\delta(\overline{M})\ge n_1$.

\begin{figure}[ht]
\centering
\begin{minipage}{0.38\textwidth}
\centering
\begin{tikzpicture}[scale=0.90]
    \def\R{1.4}
    \node[circle,draw,fill=red,inner sep=2pt,label=90:{$v_1$}] (v1) at (90:\R) {};
    \node[circle,draw,fill=red,inner sep=2pt,label=38.6:{$v_2$}] (v2) at (38.6:\R) {};
    \node[circle,draw,fill=red,inner sep=2pt,label=-115.7:{$v_5$}] (v5) at (-115.7:\R) {};
    \node[circle,draw,fill=red,inner sep=2pt,label=-167.1:{$v_6$}] (v6) at (-167.1:\R) {};
    \node[circle,draw,fill=blue,inner sep=2pt,label=-12.9:{$v_3$}] (v3) at (-12.9:\R) {};
    \node[circle,draw,fill=blue,inner sep=2pt,label=-64.3:{$v_4$}] (v4) at (-64.3:\R) {};
    \node[circle,draw,fill=blue,inner sep=2pt,label=180:{$v_7$}] (v7) at (141.4:\R) {};
    \node[circle,draw,fill=blue,inner sep=2pt,label=141.4:{$v_8$}] (v8) at (141.4:2.65) {};
    \draw (v1) -- (v2) -- (v3) -- (v4) -- (v5) -- (v6) -- (v7) -- (v1);
    \draw (v7) -- (v8);
\end{tikzpicture}
\end{minipage}
\hfill
\begin{minipage}{0.58\textwidth}
\centering
\begin{tikzpicture}[scale=0.78]
    \def\R{2.8}
    \coordinate (c1) at (90:\R);
    \coordinate (c2) at (38.6:\R);
    \coordinate (c3) at (-12.9:\R);
    \coordinate (c4) at (-64.3:\R);
    \coordinate (c5) at (-115.7:\R);
    \coordinate (c6) at (-167.1:\R);
    \coordinate (c7) at (141.4:\R);
    \coordinate (c8) at (141.4:5.3);
    \draw[very thick] (c1) -- (c7);
    \draw[very thick] (c2) -- (c3);
    \draw[very thick] (c3) -- (c4);
    \draw[very thick] (c4) -- (c5);
    \draw[very thick] (c6) -- (c7);
    \draw[very thick] (c7) -- (c8);
    \draw[dashed,very thick] (c1) -- (c2);
    \draw[dashed,very thick] (c5) -- (c6);
    \node[draw,circle,minimum size=0.7cm,fill=red!12,
          label={[label distance=4pt]90:{$B_{v_1}$}}] at (c1) {};
    \node[draw,circle,minimum size=0.7cm,fill=red!12,
          label={[label distance=4pt]38.6:{$B_{v_2}$}}] at (c2) {};
    \node[draw,circle,minimum size=0.7cm,fill=red!12,
          label={[label distance=4pt]-115.7:{$B_{v_5}$}}] at (c5) {};
    \node[draw,circle,minimum size=0.7cm,fill=red!12,
          label={[label distance=4pt]-167.1:{$B_{v_6}$}}] at (c6) {};
    \node[draw,circle,minimum size=1.2cm,fill=blue!12,
          label={[label distance=5pt]-12.9:{$B_{v_3}$}}] at (c3) {};
    \node[draw,circle,minimum size=1.2cm,fill=blue!12,
          label={[label distance=5pt]-64.3:{$B_{v_4}$}}] at (c4) {};
    \node[draw,circle,minimum size=1.2cm,fill=blue!12,
          label={[label distance=5pt]180:{$B_{v_7}$}}] at (c7) {};
    \node[draw,circle,minimum size=1.2cm,fill=blue!12,
          label={[label distance=5pt]141.4:{$B_{v_8}$}}] at (c8) {};
\end{tikzpicture}
\end{minipage}
\caption{The model graph $M$ and its pseudo-blowup $\overline{M}$ with rule $i=j$, for $k=2$.}
\label{fig:section2.1-construction}
\end{figure}
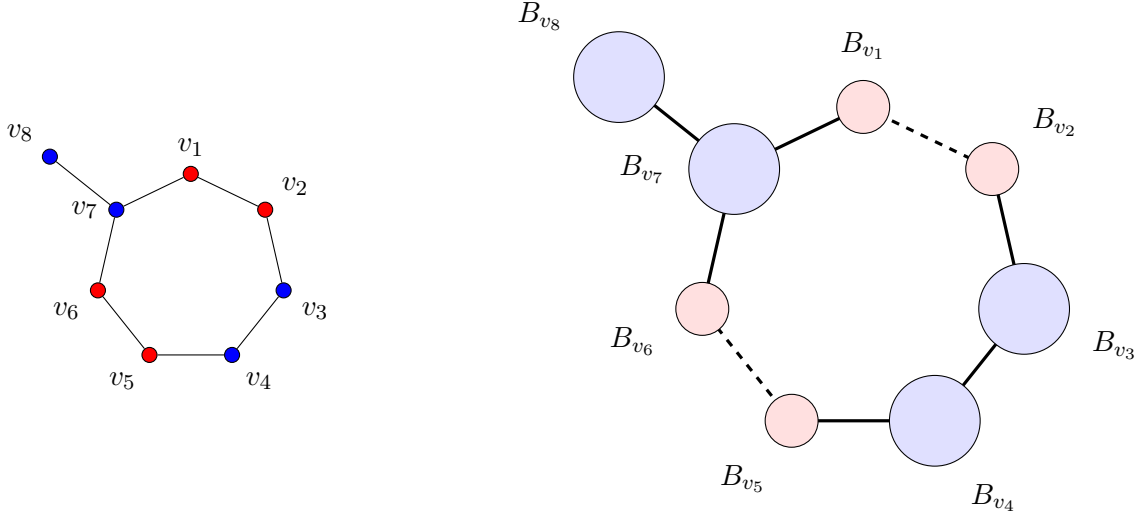

By Corollary~\ref{cor: main reduction}, it remains to verify the following construction.

\begin{lemma}\label{lemma: existence of M bar}
    Let $k \ge 2$ and $s \ge 2k$ be integers. For any graph $H \in \mathfrak C_{2s-1}^{4k-4}$, let $t = 2s \cdot |V(H)|$. Then $H \subseteq M[t]$. Furthermore, for every sufficiently large constant $C$, there exist integers $n_1 > n_2 > t \cdot C^{4k}$ and an $(n_1, n_2)$-pseudo-blowup $\overline{M}$ of $M$ such that $\delta(\overline{M}) \ge \left(\frac{1}{2k} - \frac{1}{C}\right) |V(\overline{M})|$ and  $\overline{M}$ does not contain $H$ as a subgraph.
\end{lemma}

\begin{proof}
   Since $2s-1\ge 4k-1$ are odd, $C_{2s-1}$ admits a homomorphism to
$C_{4k-1}$: take a closed walk with $(2s-1+4k-1)/2$ forward steps and
$(2s-1-(4k-1))/2$ backward steps.  Hence $C_{2s-1}\subseteq C_{4k-1}[2s]$, and
therefore
\[
    H\subseteq C_{2s-1}[|V(H)|]
    \subseteq M[2s|V(H)|]=M[t].
\]

   Choose $n_2>tC^{4k}$ and then choose $n_1>n_2$ so that
$    \frac{n_1}{2k(n_1+n_2)}\ge \frac1{2k}-\frac1C.$
Every vertex of $M$ has a neighbour in $V_1$, so every vertex of $\overline M$
is complete to some large block.  Consequently,
\[
    \delta(\overline M)\ge n_1
    \quad\text{and}\quad
    |V(\overline M)|=2k(n_1+n_2),
\]
which gives the required minimum-degree bound.

It remains to prove that $\overline M$ is $H$-free.  Suppose that
$\phi\colon H\hookrightarrow\overline M$ is an embedding.  Fix a blowup
representation of $H$ witnessing $H\in\mathfrak C_{2s-1}^{4k-4}$.  Call a
vertex of $H$ \emph{singular} if its part in this representation is a singleton,
and \emph{non-singular} otherwise.

\begin{claim}\label{claim: singular_edges}
    If an edge $uv \in E(H)$ maps to an edge between two small blocks in $\overline{M}$, then both $u$ and $v$ must be singular.
\end{claim}

\begin{poc}
    Suppose that $\phi(u)$ lies in a small block $B_1$ and $\phi(v)$ lies in the paired small block $B_2$. In the model graph $M$, every small block is adjacent to exactly one small block and one large block. Thus, the neighbours of block $B_1$ are precisely $B_2$ and some large block $B_0$.
    
    Assume for the sake of contradiction that $v$ is non-singular. Then there exists a vertex $v'$ distinct from $v$ belonging to the same part as $v$. Since $u$ is adjacent to $v$ in the blowup $H$, it must also be adjacent to $v'$. Consequently, $\phi(u)$ is adjacent to both $\phi(v)$ and $\phi(v')$ in $\overline{M}$.
    
    We first determine the location of $\phi(v')$. Since $\phi(v')$ is a neighbour of $\phi(u) \in B_1$, it must reside in either $B_2$ or $B_0$. However, $\phi(v')$ cannot be in $B_2$ because the edges between $B_1$ and $B_2$ form a perfect matching: since $\phi(u)$ is already connected to $\phi(v)$, it cannot have a second neighbour in $B_2$. Therefore, $\phi(v')$ must reside in the large block $B_0$.
    
Let $w$ lie in the other cycle part adjacent to the part containing $v$.
Then $w$ is adjacent to both $v$ and $v'$.  The only block adjacent to both
$B_0$ and $B_2$ is $B_1$, so $\phi(w)\in B_1$.  Since the pair $(B_1,B_2)$
is a matching, $\phi(v)$ has the unique neighbour $\phi(u)$ in $B_1$.
Thus $\phi(w)=\phi(u)$, contradicting injectivity.  Hence $v$ is singular,
and the same argument with $u$ and $v$ interchanged shows that $u$ is singular.
\end{poc}

We now return to the main argument. Let $C$ be a cycle of length $2s-1$ in $\phi(H)$ formed by selecting one vertex from each part. We say a vertex in $C$ is singular if and only if it corresponds to a singular vertex in $H$.

For each $i \in \{0, \dots, k-1\}$, let $E_i$ be the set of edges in $\overline{M}$ connecting the blocks $B_{v_{4i+1}}$ and $B_{v_{4i+2}}$. Observe that removing the edge $v_{4i+1}v_{4i+2}$ from $M$ results in a tree. This implies that $\overline{M} \setminus E_i$ is bipartite. Since $C$ is an odd cycle, it must contain at least one edge $e_i \in E_i$ for each $i \in \{0, \dots, k-1\}$.

By Claim~\ref{claim: singular_edges}, the $2k$ endpoints of these edges $\{e_0, \dots, e_{k-1}\}$ must be singular. By the definition of $H$, the singular vertices in $H$ form a single contiguous segment. Therefore, the segment of singular vertices in $C$ must cover the set of edges $\{e_0, \dots, e_{k-1}\}$.

List the selected edges as $e_{i_1},\ldots,e_{i_k}$ in the order in which
they occur along the singleton interval of $C$.  This interval contains $k-1$ gaps between successive
selected edges.  Put
$    S_i:=\{v_{4i+1},v_{4i+2}\}.$
A path in $\overline M$ projects to a walk in $M$.  From the cyclic placement of
the sets $S_i$, the distance in $M$ between two distinct such sets is $2$ only
for the pair $S_{k-1},S_0$, and is at least $3$ for every other pair.  Hence each
gap contains at least one internal vertex, and at most one gap contains only one;
every other gap contains at least two.  The singleton interval therefore has at
least
$    2k+1+2(k-2)=4k-3$
vertices.  This contradicts the fact that the chosen representation of
$H\in\mathfrak C_{2s-1}^{4k-4}$ has exactly $4k-4$ singleton parts.
\end{proof}

\subsection{Proof of \Cref{thm: general lowbdd for 3-color blowup}}

Let \(H\) be a graph with \(\chi(H)\ge 3\).  If \(\chi(H)\ge 4\), then
\(\delta_{\textup{B}}(H)\ge \delta_\chi(H)>0\), and the desired conclusion
follows immediately.
Hence assume that \(\chi(H)=3\).
Let $h:=|V(H)|$.  Since $\chi(H)=3$, we have $H\subseteq C_3[h]$.
Let $C_{2q+1}$ be a shortest odd cycle in $H$.  The odd girth of
$C_{2q+3}[h]$ is $2q+3$, so $H\not\subseteq C_{2q+3}[h]$.  Hence there is a
smallest $k\ge1$ such that $H\not\subseteq C_{2k+3}[h]$, and by minimality
$H\subseteq C_{2k+1}[h]$.  We prove
$\delta_{\textup{B}}(H)\ge 1/(2k+1)$.

\begin{definition}\label{def: model Mk}
    Let $k\ge 1$. The model graph $M_k$ has vertex set $\{u_1,\dots,u_{2k+1}\}\cup\{v_1,\dots,v_{2k+1}\}$ and edge set consisting of:
    \begin{enumerate}
        \item the cycle $v_1 v_2\cdots v_{2k+1}v_1$, that is, edges $\{v_iv_{i+1}: i\in[2k+1]\}$ with indices modulo $2k+1$,
        \item the matching $\{v_iu_i: i\in[2k+1]\}$,
        \item the path $u_1 u_2\cdots u_{2k+1}$, that is, edges $\{u_iu_{i+1}: i\in[2k]\}$.
    \end{enumerate}

    We partition $V(M_k)=V_1\cup V_2$ by setting $V_1=\{u_1,\dots,u_{2k+1}\}$ and $V_2=\{v_1,\dots,v_{2k+1}\}$. Vertices in $V_1$ are replaced by large blocks and vertices in $V_2$ by small blocks. The vertices in each small block are indexed by $\{1,\dots,n_2\}$.

We define two $(n_1,n_2)$-pseudo-blowups of $M_k$, differing only in the connection between small blocks along the cycle edges. In $\overline{M_k^{>}}$, a vertex with index $i$ in $B_{v_s}$ is adjacent to a vertex with index $j$ in $B_{v_{s+1}}$ if and only if $i>j$. In $\overline{M_k^{=}}$, a vertex with index $i$ in $B_{v_s}$ is adjacent to a vertex with index $j$ in $B_{v_{s+1}}$ if and only if $i=j$. Here $s$ is taken modulo $2k+1$.

\end{definition}

By Corollary~\ref{cor: main reduction}, the proof of \Cref{thm: general lowbdd for 3-color blowup} reduces to the following.

\begin{lemma}\label{lemma: Mk construction}
    Let $H$ be a graph with $\chi(H)=3$ and let $k\ge 1$ be as above. Set $t=|V(H)|$. Then $H\subseteq M_k[t]$. Moreover, for every sufficiently large $C$, there exist $n_1>n_2>t\cdot C^{4k+2}$ such that:
    \begin{enumerate}
        \item if $C_{2k+1}\subseteq H$, then $\overline{M_k^{>}}$ is $H$-free and satisfies $\delta(\overline{M_k^{>}})\ge (\frac{1}{2k+1}-\frac{1}{C})\,|V(\overline{M_k^{>}})|$\textup;
        \item if $C_{2k+1}\not\subseteq H$, then $\overline{M_k^{=}}$ is $H$-free and satisfies $\delta(\overline{M_k^{=}})\ge (\frac{1}{2k+1}-\frac{1}{C})\,|V(\overline{M_k^{=}})|$.
    \end{enumerate}
\end{lemma}

\begin{proof}
    Since $C_{2k+1}\subseteq M_k$, we have $H\subseteq C_{2k+1}[|V(H)|]\subseteq M_k[t]$.

Choose $n_2>tC^{4k+2}$ and then $n_1>n_2$ so that
$        \frac{n_1}{(2k+1)(n_1+n_2)}\ge \frac1{2k+1}-\frac1C.$
    Every vertex of $M_k$ has a neighbour in $V_1$, so both pseudo-blowups have
    minimum degree at least $n_1$ and order $(2k+1)(n_1+n_2)$.  The required
    minimum-degree bounds follow.

    It remains to show that $\overline{M_k^{>}}$ is $H$-free in Case 1 and $\overline{M_k^{=}}$ is $H$-free in Case 2.

    \medskip
    \noindent\textbf{Case 1:} $C_{2k+1}\subseteq H$. We show that $\overline{M_k^{>}}$ does not contain $C_{2k+1}$ as a subgraph, which implies that $\overline{M_k^{>}}$ is $H$-free. Suppose for contradiction that $\overline{M_k^{>}}$ contains a copy of $C_{2k+1}$.

For $i\in[2k]$, let $E_i$ consist of the edges between
$B_{v_i},B_{v_{i+1}}$ together with those between $B_{u_i},B_{u_{i+1}}$, and let
$E_{2k+1}$ consist of the edges between $B_{v_{2k+1}}$ and $B_{v_1}$.  Deleting
any $E_i$ leaves a subgraph of a blowup of a ladder, and hence a bipartite graph.
Thus an odd cycle must use an edge from every $E_i$.  A copy of $C_{2k+1}$ has
exactly $2k+1$ edges, so it uses exactly one edge from each $E_i$ and no matching
edge between $B_{v_i}$ and $B_{u_i}$.  It cannot lie in the $u$-blocks, which
induce a blowup of a path; therefore it lies entirely in
$B_{v_1},\ldots,B_{v_{2k+1}}$.  Moreover, in each block the two incident cycle
edges must share their endpoint, so the cycle uses exactly one vertex from each
$B_{v_i}$.

    Therefore, the cycle takes the form $x_1,x_2,\dots,x_{2k+1},x_1$ where $x_s\in B_{v_s}$ for each $s\in[2k+1]$. Let $i_s$ denote the index of vertex $x_s$ within block $B_{v_s}$. By the definition of $\overline{M_k^{>}}$, the edge $x_s x_{s+1}$ exists if and only if $i_s>i_{s+1}$ for $s\in[2k]$, and the edge $x_{2k+1}x_1$ exists if and only if $i_{2k+1}>i_1$. This yields the chain of inequalities
    $$i_1>i_2>i_3>\cdots>i_{2k}>i_{2k+1}>i_1,$$
    which is impossible. Therefore, $\overline{M_k^{>}}$ does not contain $C_{2k+1}$, and hence is $H$-free.

    \medskip
    \noindent\textbf{Case 2:} $C_{2k+1}\not\subseteq H$. We show that $\overline{M_k^{=}}$ is $H$-free. Suppose for contradiction that there exists an embedding $\phi\colon H\hookrightarrow \overline{M_k^{=}}$. Among all such embeddings, choose one that minimizes the number of vertices lying in the middle blocks $B_{v_2},\dots,B_{v_{2k}}$.

    First, suppose no vertex of $\phi(H)$ lies in any middle block. Then $\phi(H)$ uses only the blocks $B_{v_1}$, $B_{u_1}, B_{u_2}, \dots, B_{u_{2k+1}}$, and $B_{v_{2k+1}}$. The graph induced by these blocks is a subgraph of a blowup of $C_{2k+3}$.  Hence the embedding would give $H\subseteq C_{2k+3}[|V(H)|]$, a contradiction.
    Thus some middle block contains a vertex of $\phi(H)$. For convenience, write $v_j^{(i)}$ for the vertex in block $B_{v_j}$ with index $i$.
    
    \begin{claim}
    Fix $t\in\{2,\dots,2k\}$ and an index $i$ such that $v_t^{(i)}\in\phi(H)$. Then both edges $v_{t-1}^{(i)}v_t^{(i)}$ and $v_t^{(i)}v_{t+1}^{(i)}$ belong to $E(\phi(H))$, and consequently $v_{t-1}^{(i)}, v_{t+1}^{(i)}\in\phi(H)$.
    \end{claim}
    \begin{poc}
Let $x\in V(H)$ satisfy $\phi(x)=v_t^{(i)}$.  The neighbours of
$v_t^{(i)}$ lie in $\{v_{t-1}^{(i)},v_{t+1}^{(i)}\}\cup B_{u_t}$.  Suppose that
$v_t^{(i)}v_{t+1}^{(i)}\notin E(\phi(H))$.  Then every neighbour of $x$ is
mapped into $\{v_{t-1}^{(i)}\}\cup B_{u_t}$.  Choose an unused vertex of
$B_{u_{t-1}}$; this is possible since $n_1>|V(H)|$.  That vertex is adjacent to
$v_{t-1}^{(i)}$ and is complete to $B_{u_t}$, so moving only the image of $x$ to
it preserves every edge of $H$ and decreases the number of image vertices in the
middle blocks, a contradiction.  The other edge is forced symmetrically, by moving
$x$ to an unused vertex of $B_{u_{t+1}}$.
    \end{poc}

    Since some middle block is nonempty, there exist $t\in\{2,\dots,2k\}$ and an index $i$ with $v_t^{(i)}\in\phi(H)$. Applying the claim repeatedly, we obtain that all vertices $v_1^{(i)},v_2^{(i)},\dots,v_{2k+1}^{(i)}$ belong to $\phi(H)$ and form a path in $\phi(H)$. If the edge $v_1^{(i)}v_{2k+1}^{(i)}$ were present in $E(\phi(H))$, this path would close to form a copy of $C_{2k+1}$, contradicting $C_{2k+1}\not\subseteq H$. Therefore $v_1^{(i)}v_{2k+1}^{(i)}\notin E(\phi(H))$.
    
    Let $x_1,x_2\in V(H)$ be the preimages of $v_1^{(i)},v_2^{(i)}$.
Choose unused vertices $y_1\in B_{u_2}$ and $y_2\in B_{u_3}$, and redefine
$\phi(x_1)=y_1$ and $\phi(x_2)=y_2$.  Since
$v_1^{(i)}v_{2k+1}^{(i)}\notin E(\phi(H))$, every neighbour of $x_1$ other than
$x_2$ is mapped into $B_{u_1}$.  Every neighbour of $x_2$ other than $x_1$ is
mapped into $B_{u_2}\cup\{v_3^{(i)}\}$.  The vertex $y_1$ is complete to
$B_{u_1}$ and adjacent to $y_2$, while $y_2$ is complete to $B_{u_2}$ and adjacent
to $v_3^{(i)}$.  Thus the modified map is still an embedding, but it uses one
fewer vertex in the middle blocks, contradicting minimality.
\end{proof}


\section{Proof of non-monotonicity}\label{sec:non-monotone}

In this section, we prove \Cref{thm: non-monotone} by constructing a specific graph $H$ and showing that $\delta_{\textup{B}}(H) > \delta_{\textup{B}}(H \cup K_3)$, where $H \cup K_3$ denotes the disjoint union of $H$ and a triangle.

\begin{definition}\label{def:H-nonmonotone}
Let $H$ be the graph with vertex set $V(H) = \{v_1, v_2, v_3, v_4, v_5, v_6\}$ and edge set
$$E(H) = \{v_1v_3, v_2v_4, v_1v_5, v_1v_6, v_2v_5, v_2v_6, v_3v_5, v_3v_6, v_4v_5, v_4v_6\}.$$
Thus $H$ has $6$ vertices and $10$ edges (see \Cref{fig:H-structure}). The graph $H$ consists of a matching $\{v_1v_3, v_2v_4\}$ together with all edges between $\{v_5,v_6\}$ and $\{v_1,v_2,v_3,v_4\}$.
\end{definition}

\begin{figure}[ht]
\centering
\begin{minipage}{0.25\textwidth}
\centering
\begin{tikzpicture}[scale=0.68]
\node[circle,draw,fill=black,inner sep=2pt,label=above:{$v_2$}] (v2) at (78.75:1.5) {};
\node[circle,draw,fill=black,inner sep=2pt,label=above:{$v_1$}] (v1) at (101.25:1.5) {};
\node[circle,draw,fill=red,inner sep=2pt,label=below right:{$v_3$}] (v3) at (-41.25:1.5) {};
\node[circle,draw,fill=red,inner sep=2pt,label=below right:{$v_4$}] (v4) at (-18.75:1.5) {};
\node[circle,draw,fill=blue,inner sep=2pt,label=below left:{$v_5$}] (v5) at (198.75:1.5) {};
\node[circle,draw,fill=blue,inner sep=2pt,label=below left:{$v_6$}] (v6) at (221.25:1.5) {};
\draw[very thick] (v1) -- (v3); \draw[very thick] (v2) -- (v4);
\draw (v5) -- (v1); \draw (v5) -- (v2); \draw (v5) -- (v3); \draw (v5) -- (v4);
\draw (v6) -- (v1); \draw (v6) -- (v2); \draw (v6) -- (v3); \draw (v6) -- (v4);
\end{tikzpicture}
\subcaption{The graph $H$.}
\label{fig:H-structure}
\end{minipage}
\hfill
\begin{minipage}{0.36\textwidth}
\centering
\begin{tikzpicture}[scale=0.75]
\node[circle,draw,fill=black,inner sep=2pt,label=left:$b_1$] (b1) at (0,2) {};
\node[circle,draw,fill=black,inner sep=2pt,label=left:$b_2$] (b2) at (0,0) {};
\node[draw,ellipse,minimum width=1.5cm,minimum height=0.8cm,label=above:$B_4$] (B4) at (3,2) {};
\node[draw,ellipse,minimum width=1.5cm,minimum height=0.8cm,label=below:$B_3$] (B3) at (3,0) {};
\node[draw,ellipse,minimum width=2cm,minimum height=1cm,label=above:$B_6$] (B6) at (6,2) {};
\node[draw,ellipse,minimum width=2cm,minimum height=1cm,label=below:$B_5$] (B5) at (6,0) {};
\draw (b1) -- (B3); \draw (b1) -- (B4); \draw (b1) -- (B5);
\draw (b2) -- (B3); \draw (b2) -- (B6);
\draw[dashed,thick] (B3) -- (B4);
\draw (B3) -- (B6); \draw (B4) -- (B5); \draw (B5) -- (B6);
\end{tikzpicture}
\subcaption{The model graph $M$.}
\label{fig:M-structure}
\end{minipage}
\hfill
\begin{minipage}{0.36\textwidth}
\centering
\begin{tikzpicture}[scale=0.75]
\node[draw,ellipse,minimum width=1.5cm,minimum height=0.8cm,opacity=0.3] at (3,2) {};
\node[draw,ellipse,minimum width=1.5cm,minimum height=0.8cm,opacity=0.3] at (3,0) {};
\node[draw,ellipse,minimum width=2cm,minimum height=1cm,opacity=0.3] at (6,2) {};
\node[draw,ellipse,minimum width=2cm,minimum height=1cm,opacity=0.3] at (6,0) {};
\node[circle,draw,fill=black,inner sep=2pt,label={[label distance=1pt]above left:{$v_1$}}] (v1) at (0,2) {};
\node[circle,draw,fill=black,inner sep=2pt,label={[label distance=1pt]below left:{$v_2$}}] (v2) at (0,0) {};
\node[circle,draw,fill=red,inner sep=2pt,label={[label distance=1pt]above:{$v_3$}}] (v3) at (2.8,2) {};
\node[circle,draw,fill=blue,inner sep=2pt,label={[label distance=1pt]below:{$v_5$}}] (v5) at (2.8,-0.1) {};
\node[circle,draw,fill=blue,inner sep=2pt,label={[label distance=1pt]above:{$v_6$}}] (v6) at (3.2,0.05) {};
\node[circle,draw,fill=red,inner sep=2pt,label={[label distance=1pt]above:{$v_4$}}] (v4) at (5.8,2) {};
\draw[very thick] (v1) -- (v3); \draw[very thick] (v2) -- (v4);
\draw (v5) -- (v1); \draw (v5) -- (v2); \draw (v5) -- (v3); \draw (v5) -- (v4);
\draw (v6) -- (v1); \draw (v6) -- (v2); \draw (v6) -- (v4);
\draw[dashed,gray] (v6) -- (v3);
\end{tikzpicture}
\subcaption{A copy of $H$ in $G$.}
\label{fig:H-copy}
\end{minipage}
\caption{Illustration of the construction for the lower bound.}
\label{fig:lower-bound-construction}
\end{figure}
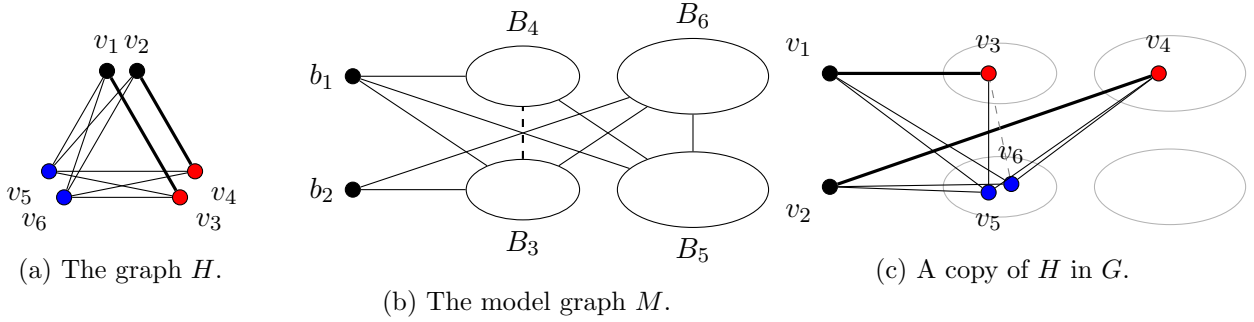

To prove \Cref{thm: non-monotone}, we show that $\delta_{\textup{B}}(H) \geq \frac{1}{2}$ and $\delta_{\textup{B}}(H \cup K_3) < \frac{1}{2}$. Since $H$ is an induced subgraph of $H \cup K_3$, this demonstrates that the blowup threshold is not monotone.

\subsection{Proof of $\delta_{\textup{B}}(H) \geq \frac{1}{2}$}\label{sec:H-lower-bound}

We construct dense maximal \(H\)-free graphs whose twin quotients
are unbounded. Let $n_1, n_2$ be positive integers with $n_1$ sufficiently large compared to $n_2$. We define a model graph $M$ (see \Cref{fig:M-structure}) with vertex set consisting of two singleton vertices $b_1, b_2$, two blocks $B_3, B_4$ of size $n_2$ each, and two blocks $B_5, B_6$ of size $n_1$ each. The vertices in blocks $B_3$ and $B_4$ are indexed by $\{1, \dots, n_2\}$. Thus $|V(M)| = 2 + 2n_2 + 2n_1$.

The edge set of $M$ is defined as follows (see \Cref{fig:M-structure}). The pairs $(B_3, B_6)$, $(B_4, B_5)$, and $(B_5, B_6)$ each induce a complete bipartite graph. Between blocks $B_3$ and $B_4$, a vertex with index $i$ in $B_3$ is adjacent to a vertex with index $j$ in $B_4$ if and only if $i = j$. The singleton $b_1$ is adjacent to all vertices in $B_3$, $B_4$, and $B_5$. The singleton $b_2$ is adjacent to all vertices in $B_3$ and $B_6$. All other pairs of vertex sets are non-adjacent.

\begin{lemma}\label{lem:H-properties}
For any $\varepsilon > 0$ and any positive integer $C$, there exist positive integers $n_1, n_2$ such that $\delta(M) \geq (\frac{1}{2} - \varepsilon)|V(M)|$, and if a graph $G$ contains $M$ as a subgraph and is also a blowup of some graph $F$ with $|V(F)| \leq C$, then $G$ contains $H$ as a subgraph.
\end{lemma}

\begin{proof}
Choose $n_2>C$, and then choose $n_1$ sufficiently large that
$    \frac{n_1}{2+2n_1+2n_2}\ge \frac{1}{2}-\varepsilon.$
Every vertex of $M$ is complete to at least one of $B_5,B_6$, so the required minimum-degree bound follows.

Suppose that $M\subseteq G$ and that $G$ is a blowup of a graph $F$ with $|V(F)|\le C$. Let
$    \pi:V(G)\to V(F)$
be the canonical projection onto the blowup parts. Since $|B_3|=n_2>C$, there are distinct vertices $x_i,x_j\in B_3$, with indices $i\ne j$, such that $\pi(x_i)=\pi(x_j)$. Let $y_i\in B_4$ be the vertex of index $i$, and choose any $z\in B_6$. Since $x_iy_i\in E(G)$ and $x_i,x_j$ lie in the same blowup part, we also have $x_jy_i\in E(G)$.

Now set
$v_1=b_1,$ $v_2=b_2,$ $v_3=y_i,$ 
$v_4=z$, $v_5=x_i$, and $v_6=x_j$.
All edges of $H$, except possibly $v_6v_3$, already belong to $M$, and the remaining edge is $x_jy_i\in E(G)$. Hence $G$ contains $H$ (see \Cref{fig:H-copy}).
\end{proof}

\begin{lemma}\label{lem:H-free}
    The graph $M$ is $H$-free.
\end{lemma}

\begin{proof}
Suppose that $M$ contains a copy of $H$. Since $M-\{b_1,b_2\}$ is bipartite, each of the two vertex-disjoint triangles
$v_5v_1v_3$
 and 
$    v_6v_2v_4$
contains one of $b_1,b_2$. Thus both $b_1,b_2$ occur in the copy.

Put $U:=\{v_5,v_6\}$. If $U=\{b_1,b_2\}$, then
$    \{v_1,v_2,v_3,v_4\}
    \subseteq N(b_1)\cap N(b_2)=B_3,$
which is impossible because $B_3$ is independent.

Suppose that exactly one vertex of $U$ belongs to $\{b_1,b_2\}$. The other vertex of $\{b_1,b_2\}$ is an endpoint of one of the two matching edges $v_1v_3,v_2v_4$. The triangle formed by the other vertex of $U$ and the other matching edge then avoids both $b_1,b_2$, contradicting the bipartiteness of $M-\{b_1,b_2\}$.

Finally, suppose that $U\cap\{b_1,b_2\}=\varnothing$. Since $b_1b_2\notin E(M)$, the vertices $b_1,b_2$ lie on different matching edges. Using an automorphism of $H$, assume $v_1=b_1$ and $v_2=b_2$. Then
$    \{v_5,v_6\}\subseteq N(b_1)\cap N(b_2)=B_3.$
The vertex $v_3$ is adjacent to $b_1$ and to both distinct vertices $v_5,v_6\in B_3$. This is impossible: among the neighbours of $b_1$, vertices of $B_3\cup B_5$ have no neighbours in $B_3$, while each vertex of $B_4$ has exactly one neighbour in $B_3$ (see \Cref{fig:H-free}).
\end{proof}

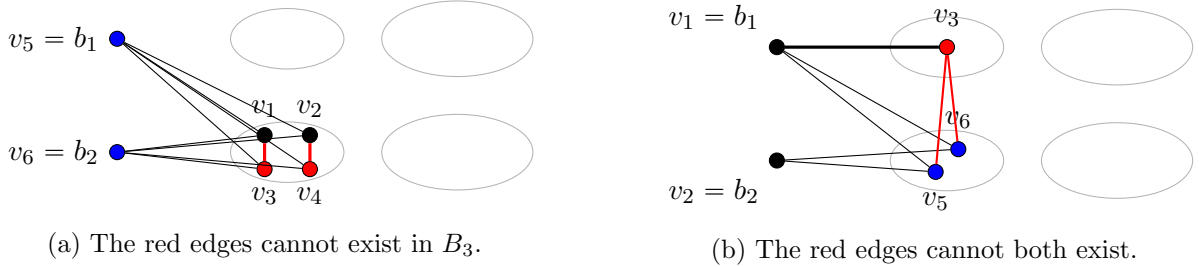
\begin{figure}[ht]
\centering
\begin{minipage}{0.48\textwidth}
\centering
\begin{tikzpicture}[scale=0.75]
\node[draw,ellipse,minimum width=1.5cm,minimum height=0.8cm,opacity=0.3] at (3,2) {};
\node[draw,ellipse,minimum width=1.5cm,minimum height=0.8cm,opacity=0.3] at (3,0) {};
\node[draw,ellipse,minimum width=2cm,minimum height=1cm,opacity=0.3] at (6,2) {};
\node[draw,ellipse,minimum width=2cm,minimum height=1cm,opacity=0.3] at (6,0) {};
\node[circle,draw,fill=blue,inner sep=2pt,label=left:{$v_5=b_1$}] (v5) at (0,2) {};
\node[circle,draw,fill=blue,inner sep=2pt,label=left:{$v_6=b_2$}] (v6) at (0,0) {};
\node[circle,draw,fill=black,inner sep=2pt,label=above:{$v_1$}] (v1) at (2.6,0.3) {};
\node[circle,draw,fill=black,inner sep=2pt,label=above:{$v_2$}] (v2) at (3.4,0.3) {};
\node[circle,draw,fill=red,inner sep=2pt,label=below:{$v_3$}] (v3) at (2.6,-0.3) {};
\node[circle,draw,fill=red,inner sep=2pt,label=below:{$v_4$}] (v4) at (3.4,-0.3) {};
\draw (v5) -- (v1); \draw (v5) -- (v2); \draw (v5) -- (v3); \draw (v5) -- (v4);
\draw (v6) -- (v1); \draw (v6) -- (v2); \draw (v6) -- (v3); \draw (v6) -- (v4);
\draw[red,very thick] (v1) -- (v3); \draw[red,very thick] (v2) -- (v4);
\end{tikzpicture}
\subcaption{The red edges cannot exist in $B_3$.}
\end{minipage}
\hfill
\begin{minipage}{0.48\textwidth}
\centering
\begin{tikzpicture}[scale=0.75]
\node[draw,ellipse,minimum width=1.5cm,minimum height=0.8cm,opacity=0.3] at (3,2) {};
\node[draw,ellipse,minimum width=1.5cm,minimum height=0.8cm,opacity=0.3] at (3,0) {};
\node[draw,ellipse,minimum width=2cm,minimum height=1cm,opacity=0.3] at (6,2) {};
\node[draw,ellipse,minimum width=2cm,minimum height=1cm,opacity=0.3] at (6,0) {};
\node[circle,draw,fill=black,inner sep=2pt,label={[label distance=1pt]above left:{$v_1=b_1$}}] (v1) at (0,2) {};
\node[circle,draw,fill=black,inner sep=2pt,label={[label distance=1pt]below left:{$v_2=b_2$}}] (v2) at (0,0) {};
\node[circle,draw,fill=red,inner sep=2pt,label={[label distance=1pt]above:{$v_3$}}] (v3) at (3,2) {};
\node[circle,draw,fill=blue,inner sep=2pt,label={[label distance=1pt]below:{$v_5$}}] (v5) at (2.8,-0.2) {};
\node[circle,draw,fill=blue,inner sep=2pt,label={[label distance=1pt]above:{$v_6$}}] (v6) at (3.2,0.2) {};
\draw (v5) -- (v1); \draw (v5) -- (v2);
\draw (v6) -- (v1); \draw (v6) -- (v2);
\draw[very thick] (v3) -- (v1);
\draw[red,thick] (v3) -- (v5); \draw[red,thick] (v3) -- (v6);
\end{tikzpicture}
\subcaption{The red edges cannot both exist.}
\end{minipage}
\caption{Illustration of why $M$ is $H$-free.}
\label{fig:H-free}
\end{figure}

\begin{theorem}\label{thm:H-lower-bound}
Let $H$ be the graph defined in Definition~\ref{def:H-nonmonotone}. Then $\delta_{\textup{B}}(H) \geq \frac{1}{2}$.
\end{theorem}
\begin{proof}
Suppose that $\delta_{\textup{B}}(H)<1/2$. Choose $\alpha<1/2$ and a graph $F$ witnessing the defining property of the blowup threshold at $\alpha$. Put
$C:=|V(F)|$
and $    \varepsilon:=\frac{1}{2}-\alpha.$
Choose $M$ as in Lemma~\ref{lem:H-properties}. By Lemma~\ref{lem:H-free}, $M$ is $H$-free. Add edges to $M$, without changing its vertex set, until obtaining a maximal $H$-free graph $G$. Then
$    \delta(G)\ge\delta(M)\ge\alpha|V(G)|,$
so $G$ is a blowup of $F$. Since $M\subseteq G$, Lemma~\ref{lem:H-properties} implies that $G$ contains $H$, a contradiction.
\end{proof}

\subsection{Proof of $\delta_{\textup{B}}(H \cup K_3) < \frac{1}{2}$}\label{subsec:upper-bound}

In this subsection, we prove $\delta_{\textup{B}}(H \cup K_3) < \frac{1}{2}$, where $H$ is the graph defined in Definition~\ref{def:H-nonmonotone}. Combined with the lower bound $\delta_{\textup{B}}(H) \geq \frac{1}{2}$ established in \Cref{sec:H-lower-bound}, this demonstrates that $\delta_{\textup{B}}(H) > \delta_{\textup{B}}(H \cup K_3)$, proving \Cref{thm: non-monotone}. The key is to show that sufficiently large maximal $(H \cup K_3)$-free graphs with minimum degree slightly below $\frac{n}{2}$ are blowups of bounded graphs, which we establish in the following theorem.

\begin{theorem}\label{thm:upper-bound-main}
There exist constants $\varepsilon > 0$, $M_0$, and $C$ such that every maximal $(H \cup K_3)$-free graph $G$ with $n \geq M_0$ vertices and minimum degree $\delta(G) \geq (\frac{1}{2} - \varepsilon)n$ is a blowup of some graph of size at most $C$.
\end{theorem}

We first show how this theorem implies the desired upper bound.

\begin{prop}\label{prop:H-union-K3-upper}
Let $H$ be the graph defined in Definition~\ref{def:H-nonmonotone}. Then $\delta_{\textup{B}}(H \cup K_3) < \frac{1}{2}$.
\end{prop}

\begin{proof}
Let $M_0,C,\varepsilon$ be given by \Cref{thm:upper-bound-main}, and put $L:=\max\{M_0,C\}$. Let $F$ be the disjoint union of one representative of every isomorphism class of graphs on at most $L$ vertices.

Let $G$ be a maximal $(H\cup K_3)$-free graph with
$    \delta(G)\ge\left(\frac{1}{2}-\varepsilon\right)|V(G)|.$
If $|V(G)|<M_0$, then $G$ is an induced subgraph of $F$. Otherwise, \Cref{thm:upper-bound-main} shows that $G$ is a blowup of a graph $F_G$ on at most $C$ vertices, and $F_G$ is an induced subgraph of $F$. Since empty blowup parts are allowed, in both cases $G$ is a blowup of $F$. Hence
$    \delta_{\textup{B}}(H\cup K_3)
    \le\frac{1}{2}-\varepsilon<\frac{1}{2}.$
\end{proof}

We now prove \Cref{thm:upper-bound-main}. Fix
$    0<\varepsilon<10^{-4},$
and then choose $M_0$ sufficiently large for all inequalities below. Let $G$ be a maximal $(H\cup K_3)$-free graph on $n\ge M_0$ vertices satisfying
$    \delta(G)\ge\left(\frac{1}{2}-\varepsilon\right)n.$
We first show that $G$ becomes bipartite after deleting a bounded set.

\begin{lemma}\label{lem:nearly-bipartite}
There exists a set $S \subseteq V(G)$ with $|S| \leq 58$ such that $G - S$ is bipartite.
\end{lemma}

\begin{proof}
Call an edge $uv$ \emph{dense} if
    $|N(u)\cap N(v)|\ge \frac{n}{10}.$
We first show that the graph formed by the dense edges has matching number at most $29$.

Suppose that $e_i=u_iv_i$, $i\in[3]$, are distinct edges in a matching, and put
$    N_i:=N(u_i)\cap N(v_i).$
If $|N_1\cap N_2\cap N_3|\ge9$, choose distinct $w_1,w_2,w_3$ in this intersection outside the six endpoints. Then
$    \{u_1,v_1,u_2,v_2,w_1,w_2\}$
contains a copy of $H$, while $\{u_3,v_3,w_3\}$ is a vertex-disjoint triangle, a contradiction. Hence
$    |N_1\cap N_2\cap N_3|\le8.$

Now suppose that $e_1,\ldots,e_{30}$ are dense edges forming a matching. For $x\in V(G)$, let
\[
    d_0(x):=|\{i:x\in N_i\}|.
\]
Then $\sum_x d_0(x)\ge3n$. If $V_3:=\{x:d_0(x)\ge3\}$, then
$    3n\le2(n-|V_3|)+30|V_3|,$
so $|V_3|\ge n/28$. On the other hand, every vertex of $V_3$ belongs to $N_i\cap N_j\cap N_k$ for some triple of indices, and therefore
$    |V_3|\le8\binom{30}{3},$
a contradiction for sufficiently large $n$.

Let $\mathcal M$ be a maximal matching of dense edges and let $S$ be the set of its endpoints. Then $|S|\le58$, and the maximality of $\mathcal M$ implies that $G-S$ contains no dense edge.

If $abc$ were a triangle in $G-S$, then
\[
|N(a)\cap N(b)|+|N(a)\cap N(c)|+|N(b)\cap N(c)|\ge d(a)+d(b)+d(c)-n
\ge\left(\frac{1}{2}-3\varepsilon\right)n.
\]
Thus one edge of the triangle would be dense, a contradiction. Hence $G-S$ is triangle-free. Moreover,
$    \delta(G-S)
    \ge\left(\frac{1}{2}-\varepsilon\right)n-58
    >\frac25|V(G-S)|.$
The Andr\'asfai--Erd\H{o}s--S\'os theorem now implies that $G-S$ is bipartite.
\end{proof}

Among all vertex sets $S$ such that $G - S$ is bipartite, we choose one of minimum size and denote it by $Z$. By the above lemma, such a set exists with $|Z| \leq 58$. The two parts of the bipartite graph $G - Z$ are denoted by $A$ and $B$.

We have now established the basic structure of $G$: it decomposes as $G = A \cup B \cup Z$, where $A$ and $B$ are independent sets, $|Z| \leq 58$, and $Z$ is chosen to be minimal such that $G - Z$ is bipartite. To show that $G$ is a blowup, we analyze the connections between $Z$ and the two parts $A$ and $B$.

Before proceeding, we establish 
two facts that will be used throughout the proof. The first states the structural properties of the partition $G = A\cup B\cup Z$. The second will be our main tool for concluding that $G$ is a blowup in all subsequent arguments.

\begin{fact}\label{fact:Z-minimality-and-degree-bounds}
    For the partition $G = A\cup B\cup Z$ as above, we have the following properties.
    \begin{itemize}
        \item Every vertex in $Z$ has neighbours in both $A$ and $B$.
        \item We have $|A|, |B| \in [(\frac{1}{2} - 2\varepsilon)n, (\frac{1}{2} + 2\varepsilon)n]$, and each vertex in $A$ has at least $(\frac{1}{2} - 2\varepsilon)n$ neighbours in $B$, and similarly each vertex in $B$ has at least $(\frac{1}{2} - 2\varepsilon)n$ neighbours in $A$.
    \end{itemize}
\end{fact}

\begin{proof}
Suppose that some $z\in Z$ has no neighbours in $A$, and put
$    Z':=Z\setminus\{z\}.$
Then
$    A\cup\{z\}$
and     $B$
form a bipartition of $G-Z'$, contradicting the minimality of $|Z|$.
The case where $z$ has no neighbours in $B$ is symmetric. This proves the first statement in Fact~\ref{fact:Z-minimality-and-degree-bounds}.

For the second statement, consider a vertex $v \in A$. Since
$\delta(G) \geq \left(\frac{1}{2} - \varepsilon\right)n$
and     $|Z| \leq 58,$
we have
$    |N(v)\cap(A\cup B)|
    \ge
    \left(\frac{1}{2}-\varepsilon\right)n-58
    \ge
    \left(\frac{1}{2}-2\varepsilon\right)n.$
Since $A$ is independent,
$    |N(v)\cap B|
    \ge
    \left(\frac{1}{2}-2\varepsilon\right)n.$
Consequently,
$|B|\ge\left(\frac{1}{2}-2\varepsilon\right)n.$
By symmetry,
$    |A|\ge\left(\frac{1}{2}-2\varepsilon\right)n,$
and every vertex in $B$ has at least
$    \left(\frac{1}{2}-2\varepsilon\right)n$
neighbours in $A$. Since $|A|+|B|\le n$, it follows that
$    |A|,|B|
    \in
    \left[
        \left(\frac{1}{2}-2\varepsilon\right)n,
        \left(\frac{1}{2}+2\varepsilon\right)n
    \right].$
\end{proof}

\begin{fact}\label{fact:extension-graph}
Let $G$ be a maximal $(H \cup K_3)$-free graph, and let $G'$ be a graph satisfying the following conditions:
\begin{enumerate}[(1)]
\item $V(G) = V(G')$ and $E(G) \subseteq E(G')$,
\item there exists a finite vertex set $S \subseteq V(G')$ such that $G' - S$ is a complete bipartite graph,
\item $G'$ is $(H \cup K_3)$-free.
\end{enumerate}
Then $G$ is a blowup of a graph of size at most $|S| + 2^{|S|+1}$.
\end{fact}

\begin{proof}
By condition (1),
$ V(G)=V(G')$
and     $E(G)\subseteq E(G').$
Since $G$ is maximal $(H\cup K_3)$-free and $G'$ is $(H\cup K_3)$-free, maximality implies that
    $G=G'.$

Let
$    S=\{s_1,\ldots,s_k\},$
and let $V_1,V_2$ be the two parts of the complete bipartite graph $G'-S$. For each vertex $v\in V_1\cup V_2$, define its connection pattern to $S$ by
\[
    \bigl(\mathbbm 1_{vs_1\in E(G')},\ldots,
          \mathbbm 1_{vs_k\in E(G')}\bigr)\in\{0,1\}^k.
\]
Partition each of $V_1$ and $V_2$ according to these patterns. Together with the singleton parts corresponding to $S$, this gives at most
    $k+2\cdot2^k
    =
    |S|+2^{|S|+1}$
parts.

Each part is independent. Between two parts contained in opposite sides of $G'-S$, all edges are present; between two parts contained in the same side, no edges are present; and adjacency between a singleton in $S$ and any other part is constant by the definition of the connection patterns. Hence $G=G'$ is a blowup of a graph on at most
$    |S|+2^{|S|+1}$
vertices.
\end{proof}

We first handle the case $|Z| \leq 2$. We construct a graph $G'$ with $V(G') = V(G)$ and $E(G') = E(G) \cup E'$, where $E'$ consists of all edges between $A$ and $B$ that are not in $G$. By this construction, $G' - Z$ is a complete bipartite graph with parts $A$ and $B$. It suffices to verify that $G'$ is $(H \cup K_3)$-free. Note that $H \cup K_3$ contains three vertex-disjoint triangles. If $G'$ contains $H \cup K_3$, then after removing the at most $2$ vertices of $Z$, at least one complete triangle remains in $G' - Z = A \cup B$. But $A \cup B$ is bipartite, a contradiction. By Fact~\ref{fact:extension-graph}, $G$ is a blowup of a graph of size at most $2 + 2^3 = 10$.

For the remainder of the proof, we assume $|Z| \geq 3$. Choose three vertices $u,v,w\in Z$. Inclusion--exclusion gives
\[
|N(u)\cap N(v)|+|N(u)\cap N(w)|+|N(v)\cap N(w)|\ge d(u)+d(v)+d(w)-n
\ge\left(\frac{1}{2}-3\varepsilon\right)n.
\]
Thus some pair, say $u,v$, has at least $(1/6-\varepsilon)n$ common neighbours. Since $|Z|\le58$, one of $A,B$ contains at least $9\varepsilon n$ common neighbours of $u,v$.

We have established that there exist $u, v \in Z$ such that either $|N(u) \cap N(v) \cap A| \geq 9\varepsilon n$ or $|N(u) \cap N(v) \cap B| \geq 9\varepsilon n$. We split into two cases:
\begin{itemize}
\item \textbf{Case 1}: There exist $u, v \in Z$ such that $|N(u) \cap N(v) \cap A| \geq 9\varepsilon n$ and $|(N(u) \cup N(v)) \cap B| \geq 2$, or $|N(u) \cap N(v) \cap B| \geq 9\varepsilon n$ and $|(N(u) \cup N(v)) \cap A| \geq 2$.
\item \textbf{Case 2}: For any $u, v \in Z$ with $|N(u) \cap N(v) \cap A| \geq 9\varepsilon n$, we have $|(N(u) \cup N(v)) \cap B| \leq 1$, and for any $u, v \in Z$ with $|N(u) \cap N(v) \cap B| \geq 9\varepsilon n$, we have $|(N(u) \cup N(v)) \cap A| \leq 1$.
\end{itemize}
Exactly one of these two cases holds. We analyze each case separately in the following subsections. The key step in both cases is to show that the maximal $(H \cup K_3)$-free condition forces $G$ to have a simple structure, which then allows us to construct a graph $G'$ satisfying the conditions of Fact~\ref{fact:extension-graph}.

\subsubsection{The case of multiple connections}\label{sec:the-case-of-multiple-connections}
By symmetry, we may assume there exist $z_1, z_2 \in Z$ such that $|N(z_1) \cap N(z_2) \cap A| \geq 9\varepsilon n$ and $|(N(z_1) \cup N(z_2)) \cap B| \geq 2$.
Both $N(z_1)\cap B$ and $N(z_2)\cap B$ are nonempty, and their union has size at least $2$. Hence there are distinct $b_1,b_2\in B$ such that
    $z_1b_1,z_2b_2\in E(G).$

Let $Z' = Z \setminus \{z_1, z_2\}$. Since $|Z| \geq 3$, we have $Z' \neq \varnothing$. We establish three structural properties of $Z'$ and the pair $(z_1,z_2)$.

\begin{fact}\label{fact:Z'-restricted-Z'-independent-z1-z2-restricted}
    Under the assumptions of this subsubsection, we have 
    \begin{enumerate}[(1)]
        \item\label{item:Z'-restricted} 
        For any distinct vertices $b_1, b_2 \in B$ such that $z_1 b_1, z_2 b_2 \in E(G)$, we have $N(z) \cap B \subseteq \{b_1, b_2\}$ for all $z \in Z'$.
        
        \item\label{item:Z'-independent}
        $Z'$ is an independent set.

        \item\label{item:z1-z2-restricted}
        For any distinct vertices $b_1, b_2 \in B$ such that $z_1 b_1, z_2 b_2 \in E(G)$, at least one of $z_1, z_2$ satisfies $N(z_i) \cap B \subseteq \{b_1, b_2\}$.
    \end{enumerate}
\end{fact}

\begin{proof}
For the proof of \ref{item:Z'-restricted}, fix distinct $b_1,b_2\in B$ such that
$    z_1b_1,z_2b_2\in E(G).$
Suppose that there exist
$    z_3\in Z'$
and 
$    b_3\in B\setminus\{b_1,b_2\}$
such that $z_3b_3\in E(G)$. Since $z_3$ has at most $|Z|-1$ neighbours in $Z$,
    $|N(z_3)\cap(A\cup B)|
    \ge
    \delta(G)-(|Z|-1)
    \ge
    \left(\frac{1}{2}-2\varepsilon\right)n.$
Hence either
$|N(z_3)\cap A|
    \ge
    \left(\frac{1}{4}-\varepsilon\right)n$
or
$    |N(z_3)\cap B|
    \ge
    \left(\frac{1}{4}-\varepsilon\right)n.$
In the first case,
$|N(z_3)\cap N(b_3)\cap A|
    \ge
    \left(\frac{1}{4}-5\varepsilon\right)n,$
so $z_3b_3$ lies in a triangle. In the second case, choose
    $a_3\in N(z_3)\cap A.$
Then
    $|N(z_3)\cap N(a_3)\cap B|
    \ge
    \left(\frac{1}{4}-5\varepsilon\right)n,$
so there is
   $ c_3\in
    \bigl(N(z_3)\cap N(a_3)\cap B\bigr)
    \setminus\{b_1,b_2\}.$
Thus in either case there are $a_3\in A$ and
$c_3\in B\setminus\{b_1,b_2\}$ such that $z_3a_3c_3$ is a triangle.

Set
$A^*
    :=
    N(z_1)\cap N(z_2)\cap N(b_1)\cap N(b_2)\cap A.$
Then
\[
\begin{aligned}
|A^*|
&\ge
|N(z_1)\cap N(z_2)\cap A|
+|N(b_1)\cap A|
+|N(b_2)\cap A|
-2|A|\\
&\ge
9\varepsilon n
+2\left(\frac{1}{2}-2\varepsilon\right)n
-2\left(\frac{1}{2}+2\varepsilon\right)n=
\varepsilon n.
\end{aligned}
\]
Choose distinct
$  a_1,a_2\in A^*\setminus\{a_3\}.$
The vertices
$    \{z_1,b_1,z_2,b_2,a_1,a_2\}$
form a copy of $H$, while
$    \{z_3,c_3,a_3\}$
form a vertex-disjoint triangle. This contradicts that $G$ is $(H\cup K_3)$-free, proving \ref{item:Z'-restricted}.

For the proof of \ref{item:Z'-independent}, suppose that
    $z_3z_4\in E(G)$
for some distinct $z_3,z_4\in Z'$. By \ref{item:Z'-restricted}, for $i=3,4$,
$|N(z_i)\cap B|\le2$, and consequently,
   $|N(z_i)\cap A|
    \ge
    \delta(G)-2-(|Z|-1)
    \ge
    \left(\frac{1}{2}-2\varepsilon\right)n$.
 It follows that
$|N(z_3)\cap N(z_4)\cap A|
    \ge
    \left(\frac{1}{2}-6\varepsilon\right)n.$
Choose
$    a_3\in N(z_3)\cap N(z_4)\cap A.$
Then $z_3z_4a_3$ is a triangle. Since $|A^*|\ge\varepsilon n$, choose distinct
$    a_1,a_2\in A^*\setminus\{a_3\}.$
The vertices
$    \{z_1,b_1,z_2,b_2,a_1,a_2\}$
form a copy of $H$, while
$    \{z_3,z_4,a_3\}$
form a vertex-disjoint triangle, a contradiction. Thus $Z'$ is independent.

For the proof of \ref{item:z1-z2-restricted}, suppose that both $z_1$ and $z_2$ have neighbours in $B\setminus\{b_1,b_2\}$. Choose
$b_3,b_4\in B\setminus\{b_1,b_2\}$
such that
$    z_1b_3,z_2b_4\in E(G).$ Applying \ref{item:Z'-restricted} to the three pairs
$(b_1,b_2)$, $(b_3,b_2)$, $
    (b_1,b_4),$
gives, for every $z\in Z'$,
$N(z)\cap B
\subseteq
\{b_1,b_2\}
\cap
\{b_3,b_2\}
\cap
\{b_1,b_4\}=
\varnothing.$
Since $Z'\neq\varnothing$, this contradicts the fact that every vertex of $Z$ has a neighbour in $B$. This proves \ref{item:z1-z2-restricted} and completes the proof.
\end{proof}

We now prove \cref{thm:upper-bound-main} by constructing an extension graph $G'$ and apply Fact~\ref{fact:extension-graph}.

\begin{proof}[Proof of Theorem~\ref{thm:upper-bound-main} in Case 1]
    Let $Z' = Z \setminus \{z_1, z_2\}$ and $B' = B \setminus \{b_1, b_2\}$. We construct $G'$ with $V(G') = V(G)$ and $E(G') = E(G) \cup E'$, where $E'$ consists of all missing edges between $Z'$ and $A$, and all missing edges between $A$ and $B'$. Then $G' - \{z_1, z_2, b_1, b_2\}$ is a complete bipartite graph with parts $A$ and $B' \cup Z'$, since $B'$ and $Z'$ are independent sets and $Z'$ has no neighbours in $B'$ by Fact~\ref{fact:Z'-restricted-Z'-independent-z1-z2-restricted}~\ref{item:Z'-restricted}.
By Fact~\ref{fact:Z'-restricted-Z'-independent-z1-z2-restricted}~\ref{item:z1-z2-restricted}, at least one of $z_1, z_2$ satisfies $N(z_i) \cap B \subseteq \{b_1, b_2\}$. Without loss of generality, assume $N(z_2) \cap B \subseteq \{b_1, b_2\}$, that is, $N(z_2) \cap B' = \varnothing$. 

By Fact~\ref{fact:extension-graph} it suffices to verify that $G'$ is $(H \cup K_3)$-free. We proceed in two cases: when $z_1$ has neighbours in $B'$, and when it does not.

\medskip

\noindent\textbf{Case 1a}. We first consider the case when $z_1$ has neighbours in $B'$. Since $N(z_1) \cap B' \neq \varnothing$, there exists $b_3 \in B' = B \setminus \{b_1, b_2\}$ such that $z_1 b_3 \in E(G)$. We show the following two claims.

\begin{claim}\label{claim:subcase1a-Z'}
For any $z \in Z'$, we have $N(z) \cap B = \{b_2\}$.
\end{claim}

\begin{poc}
By Fact~\ref{fact:Z'-restricted-Z'-independent-z1-z2-restricted}~\ref{item:Z'-restricted} applied to the pair $(b_1, b_2)$, we have $N(z) \cap B \subseteq \{b_1, b_2\}$. Now consider the pair $(b_3, b_2)$. Since $b_3 \neq b_2$ and $z_1 b_3, z_2 b_2 \in E(G)$, by Fact~\ref{fact:Z'-restricted-Z'-independent-z1-z2-restricted}~\ref{item:Z'-restricted}, we have $N(z) \cap B \subseteq \{b_3, b_2\}$. Taking the intersection, $N(z) \cap B \subseteq \{b_1, b_2\} \cap \{b_3, b_2\} = \{b_2\}$. By Fact~\ref{fact:Z-minimality-and-degree-bounds}, $z$ has neighbours in $B$, so $N(z) \cap B = \{b_2\}$.
\end{poc}

\begin{claim}\label{claim:subcase1a-z2}
$N(z_2) \cap B = \{b_2\}$.
\end{claim}

\begin{poc}
By assumption, $N(z_2) \cap B \subseteq \{b_1, b_2\}$. Suppose for contradiction that $z_2 b_1 \in E(G)$. Consider the pair $(b_1, b_3)$. Since $b_1\neq b_3$ and $z_2b_1,z_1b_3\in E(G)$, Fact~\ref{fact:Z'-restricted-Z'-independent-z1-z2-restricted}~\ref{item:Z'-restricted}, applied after interchanging $z_1$ and $z_2$, gives
    $N(z)\cap B\subseteq\{b_1,b_3\}$
     for every $z\in Z'$.
But by Claim~\ref{claim:subcase1a-Z'}, $N(z) \cap B = \{b_2\}$. This implies $b_2 \in \{b_1, b_3\}$, contradicting $b_2 \neq b_1$ and $b_2 \neq b_3$. Therefore $z_2 b_1 \notin E(G)$, so $N(z_2) \cap B = \{b_2\}$.
\end{poc}

By Claim~\ref{claim:subcase1a-Z'}, all vertices in $Z'$ have neighbours in $B$ restricted to $\{b_2\}$, and we show that $|Z'| = 1$. Consider the graph $G - \{z_1, z_2, b_2\}$. Since vertices in $Z'$ only connect to $b_2$ in $B$, after removing $b_2$, they have no neighbours in $B' \cup \{b_1\}$. Moreover, $Z'$ is an independent set by Fact~\ref{fact:Z'-restricted-Z'-independent-z1-z2-restricted}~\ref{item:Z'-independent}, and $B' \cup \{b_1\}$ is an independent set as a subset of $B$. Therefore, the vertices of $G - \{z_1, z_2, b_2\}$ can be partitioned into two parts: $A$ and $(B' \cup \{b_1\}) \cup Z'$, where both parts are independent sets. This shows that $G - \{z_1, z_2, b_2\}$ is bipartite. By the minimality of $Z$, we have $|Z| = 3$, and thus $|Z'| = 1$. Let $Z' = \{z_3\}$.

We now verify that $G'$ is $(H \cup K_3)$-free. Suppose for contradiction that $G'$ contains a copy of $H \cup K_3$. Since $H \cup K_3$ contains three vertex-disjoint triangles, after removing $z_1$ from $G'$, at least two vertex-disjoint triangles remain in $G' - \{z_1\}$. We show this is impossible.

The vertices of $G' - \{z_1\}$ can be partitioned into two parts: $A$ and $B' \cup \{b_1, b_2, z_2, z_3\}$. Since $A$ is an independent set and $B' \cup \{b_1, b_2\}$ is an independent set (as a subset of $B$), the graph $G' - \{z_1\}$ is a bipartite graph between these two parts, together with at most three additional edges within the second part: $z_2 z_3$, $z_2 b_2$, and $z_3 b_2$. Since the bipartite graph itself contains no triangles, any triangle in $G' - \{z_1\}$ must contain at least one of these three edges. Therefore, two vertex-disjoint triangles must contain at least two distinct edges from $\{z_2 z_3, z_2 b_2, z_3 b_2\}$. However, any two edges from this set share at least one vertex, so the two triangles cannot be vertex-disjoint, a contradiction. Therefore $G'$ is $(H \cup K_3)$-free. By Fact~\ref{fact:extension-graph}, $G$ is a blowup of a graph of size at most $4 + 2^5 = 36$, completing the proof.

\medskip

\noindent\textbf{Case 1b}. We now consider the case when both $z_1$ and $z_2$ have no neighbours in $B'$. By Fact~\ref{fact:Z'-restricted-Z'-independent-z1-z2-restricted}~\ref{item:Z'-restricted}, all vertices in $Z$ have neighbours in $B$ restricted to $\{b_1, b_2\}$.

For every $z\in Z$, all neighbours outside $A$ lie in
    $(N(z)\cap B)\cup(Z\setminus\{z\}).$
Consequently,
\[
    |N(z)\cap A|
    \ge\delta(G)-2-(|Z|-1)
    \ge\left(\frac{1}{2}-2\varepsilon\right)n.
\]
The same lower bound holds for $|N(b_i)\cap A|$, $i\in[2]$, because $B$ is independent and $|Z|\le58$.

\begin{claim}\label{claim:subcase1b-no-matching}
$Z \cup \{b_1, b_2\}$ contains no matching of size $3$.
\end{claim}

\begin{poc}
Suppose that $u_iv_i$, $i\in[3]$, form a matching in $G[Z\cup\{b_1,b_2\}]$. Each of the six endpoints has at least $(1/2-2\varepsilon)n$ neighbours in $A$. Since $|A|\le(1/2+2\varepsilon)n$, their common neighbourhood in $A$ has size at least
    $6\left(\frac{1}{2}-2\varepsilon\right)n
    -5\left(\frac{1}{2}+2\varepsilon\right)n
    =\left(\frac{1}{2}-22\varepsilon\right)n.$
Choose three distinct vertices $a_1,a_2,a_3$ in this common neighbourhood. Then
$    \{u_1,v_1,u_2,v_2,a_1,a_2\}$
contains a copy of $H$, while $\{u_3,v_3,a_3\}$ is a vertex-disjoint triangle, a contradiction.
\end{poc}

We now verify that $G'$ is $(H \cup K_3)$-free. Recall that $G'$ is constructed by adding all missing edges between $Z'$ and $A$, and between $A$ and $B'$. In this case, since all vertices in $Z$ have neighbours in $B$ restricted to $\{b_1, b_2\}$, there are no edges between $Z \cup \{b_1, b_2\}$ and $B'$ in $G'$.

\begin{claim}\label{claim:subcase1b-triangle-structure}
Every triangle in $G'$ contains an edge within $Z \cup \{b_1, b_2\}$.
\end{claim}

\begin{poc}
Consider any triangle in $G'$. If the triangle contains a vertex from $B'$, then since $B'$ is an independent set and $B'$ has no edges to $Z \cup \{b_1, b_2\}$, the other two vertices of the triangle must both be in $A$. But $A$ is also an independent set, so this is impossible. Therefore, every triangle has all its vertices in $A \cup Z \cup \{b_1, b_2\}$. Since $A$ is an independent set, the triangle must contain at least two vertices from $Z \cup \{b_1, b_2\}$, and these two vertices must be adjacent.
\end{poc}

Suppose for contradiction that $G'$ contains a copy of $H \cup K_3$. Since $H \cup K_3$ contains three vertex-disjoint triangles, by Claim~\ref{claim:subcase1b-triangle-structure}, each triangle contains an edge within $Z \cup \{b_1, b_2\}$. Therefore, the three vertex-disjoint triangles require a matching of size $3$ in $Z \cup \{b_1, b_2\}$, contradicting Claim~\ref{claim:subcase1b-no-matching}. Therefore $G'$ is $(H \cup K_3)$-free. By Fact~\ref{fact:extension-graph}, $G$ is a blowup of a graph of size at most $4 + 2^5 = 36$, completing the proof.
\end{proof}

\subsubsection{The case of unique common neighbour}
In this subsubsection, we prove Theorem~\ref{thm:upper-bound-main} in Case 2.

\begin{proof}[Proof of Theorem~\ref{thm:upper-bound-main} in Case 2]
By the conclusion preceding the case split, there exist $u,v\in Z$ such that
one of
$    |N(u)\cap N(v)\cap A|\ge 9\varepsilon n,$
$    |N(u)\cap N(v)\cap B|\ge 9\varepsilon n$
holds.  Swapping $A$ and $B$ if necessary, we may assume that
there are $z_1,z_2\in Z$ with
\[
    |N(z_1)\cap N(z_2)\cap A|\ge9\varepsilon n.
\]
The Case 2 assumption gives
$    |(N(z_1)\cup N(z_2))\cap B|\le1.$
Both vertices have a neighbour in $B$, so there is a unique $b\in B$ such that
$    N(z_1)\cap B=N(z_2)\cap B=\{b\}.$
Since each $z_i$ has at most $58$ neighbours outside $A$,
$    |N(z_i)\cap A|
    \ge\left(\frac{1}{2}-2\varepsilon\right)n$, $i\in[2]$.
Partition $Z$ as
\[
    Z_B:=\{z\in Z:|N(z)\cap A|\ge16\varepsilon n\},
    \qquad
    Z_A:=Z\setminus Z_B.
\]

\begin{claim}\label{claim:Z1-restricted}
For any $z \in Z_B$, we have $N(z) \cap B = \{b\}$.
\end{claim}

\begin{poc}
Since $z \in Z_B$, we have $|N(z) \cap A| \geq 16\varepsilon n$. Since $z_1$ has at least $(\frac{1}{2} - 2\varepsilon)n$ neighbours in $A$, we have $|N(z) \cap N(z_1) \cap A| \geq |N(z) \cap A| + |N(z_1) \cap A| - |A| \geq 16\varepsilon n + (\frac{1}{2} - 2\varepsilon)n - (\frac{1}{2} + 2\varepsilon)n \geq 9\varepsilon n$. By the case assumption, $|(N(z) \cup N(z_1)) \cap B| \leq 1$. Since $N(z_1) \cap B = \{b\}$ and $z$ has neighbours in $B$ by Fact~\ref{fact:Z-minimality-and-degree-bounds}, we have $N(z) \cap B = \{b\}$.
\end{poc}

\begin{claim}\label{claim:common-neighbours}
Let
$    X:=Z_B\cup B$ and $    Y:=Z_A\cup A$.
Every set of at most four vertices in $X$ has at least $10$ common neighbours in $A$, and every set of at most four vertices in $Y$ has at least $10$ common neighbours in $B$.
\end{claim}

\begin{poc}
Every vertex of $B$ has at least $(1/2-2\varepsilon)n$ neighbours in $A$. If $z\in Z_B$, then Claim~\ref{claim:Z1-restricted} gives $N(z)\cap B=\{b\}$, and hence
\[
    |N(z)\cap A|
    \ge\delta(G)-1-(|Z|-1)
    \ge\left(\frac{1}{2}-2\varepsilon\right)n.
\]
Thus every vertex of $X$ has at least $(1/2-2\varepsilon)n$ neighbours in $A$.

Every vertex of $A$ has at least $(1/2-2\varepsilon)n$ neighbours in $B$. If $z\in Z_A$, then
\[
    |N(z)\cap B|
    \ge\delta(G)-|N(z)\cap A|-(|Z|-1)
    \ge\left(\frac{1}{2}-18\varepsilon\right)n.
\]
Thus every vertex of $Y$ has at least $(1/2-18\varepsilon)n$ neighbours in $B$.

For $1\le r\le4$, inclusion--exclusion and
$|A|,|B|\le(1/2+2\varepsilon)n$ now give
\[
    \left|A\cap\bigcap_{i=1}^rN(x_i)\right|
    \ge\left(\frac{1}{2}-14\varepsilon\right)n \quad \text{and }\quad  \left|B\cap\bigcap_{i=1}^rN(y_i)\right|
    \ge\left(\frac{1}{2}-78\varepsilon\right)n
\]
for $x_1,\ldots,x_r\in X$ and $y_1,\ldots,y_r\in Y$. Both quantities are at least $10$ for sufficiently large $n$.
\end{poc}

To show that $G$ is a blowup, we construct an extension graph $G'$ and apply Fact~\ref{fact:extension-graph}. Define $G'$ with $V(G') = V(G)$ and $E(G') = E(G) \cup E'$, where $E'$ consists of all missing edges between $(Z_B \cup B)$ and $(Z_A \cup A)$. Then $G' - Z$ is a complete bipartite graph with parts $B$ and $A$. By Fact~\ref{fact:extension-graph}, it remains to verify that $G'$ is $(H \cup K_3)$-free. Let $E_0 = \{uv \in E(G) : u, v \in Z_B \cup B\} \cup \{uv \in E(G) : u, v \in Z_A \cup A\}$ denote the edges of $G$ with both endpoints in the same part.

\begin{claim}\label{claim:no-3-matching}
$E_0$ contains no matching of size $3$.
\end{claim}

\begin{poc}
Suppose that $e_i=u_iv_i$, $i\in[3]$, form a matching in $E_0$. At least two lie in the same one of $X,Y$; assume $e_1,e_2\subseteq X$. By Claim~\ref{claim:common-neighbours}, choose distinct
\[
    a_1,a_2\in
    N(u_1)\cap N(v_1)\cap N(u_2)\cap N(v_2)\cap A
\]
away from the matching endpoints. Apply the same claim to $u_3,v_3$ and choose a common neighbour $x$ outside the eight already selected vertices. Then
$    \{u_1,v_1,u_2,v_2,a_1,a_2\}$
contains a copy of $H$, while $\{u_3,v_3,x\}$ is a vertex-disjoint triangle, a contradiction.
\end{poc}

\begin{claim}\label{claim:G'-H-K3-free}
$G'$ is $(H \cup K_3)$-free.
\end{claim}

\begin{poc}
Suppose for contradiction that $G'$ contains a copy of $H \cup K_3$. Since $G' - E_0$ is a complete bipartite graph with parts $(Z_B \cup B)$ and $(Z_A \cup A)$ (which contains no triangles), each triangle in $G'$ must contain at least one edge from $E_0$. Therefore, the three vertex-disjoint triangles require a matching of size $3$ in $E_0$, contradicting Claim~\ref{claim:no-3-matching}.
\end{poc}

By Fact~\ref{fact:extension-graph}, $G$ is a blowup of a graph of size at most $|Z| + 2^{|Z|+1} \leq 58 + 2^{59}$. This completes the proof of \Cref{thm:upper-bound-main}, and hence the proof of Proposition~\ref{prop:H-union-K3-upper}. Combined with \Cref{thm:H-lower-bound}, this establishes \Cref{thm: non-monotone}.
\end{proof}

\section{Proof of the exact value \texorpdfstring{$\frac{1}{4}$}{1/4}}\label{sec:proof-half-k}

In this section, we prove the upper bound \Cref{thm: uppbdd for 1/2k}.
Together with the lower bound \Cref{thm: lowbdd for 1/2k}, this proves
\Cref{thm:existence of 1/2k}.


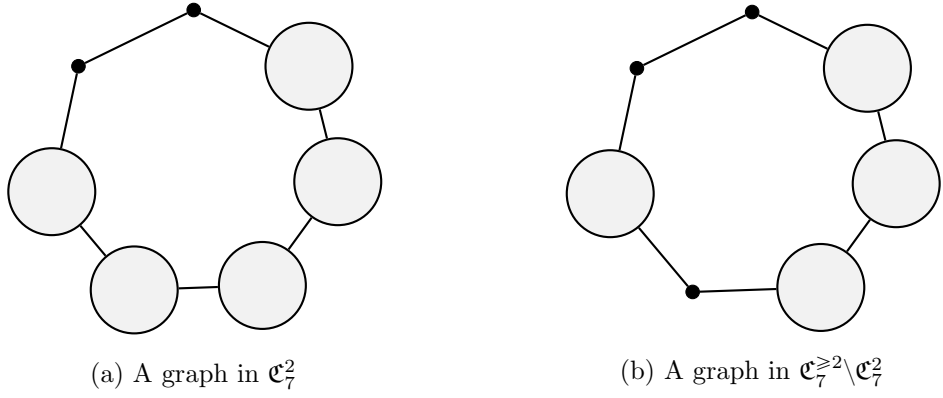
\begin{figure}[ht]
\centering

\begin{minipage}{0.42\textwidth}
\centering
\begin{tikzpicture}[
    scale=0.9,
    every node/.style={font=\small},
    singleton/.style={circle, draw, fill=black, inner sep=1.8pt},
    blowup/.style={circle, draw, thick, minimum size=11.5mm, fill=gray!10},
    edge/.style={thick}
]

\node[singleton] (v1) at (90:2.15) {};
\node[singleton] (v2) at (142:2.15) {};
\node[blowup]    (v3) at (194:2.15) {};
\node[blowup]    (v4) at (246:2.15) {};
\node[blowup]    (v5) at (298:2.15) {};
\node[blowup]    (v6) at (350:2.15) {};
\node[blowup]    (v7) at (38:2.15)  {};

\draw[edge] (v1) -- (v2);
\draw[edge] (v2) -- (v3);
\draw[edge] (v3) -- (v4);
\draw[edge] (v4) -- (v5);
\draw[edge] (v5) -- (v6);
\draw[edge] (v6) -- (v7);
\draw[edge] (v7) -- (v1);

\end{tikzpicture}

\vspace{1.5mm}
{\small (a) A graph in $\mathfrak C_7^{2}$}
\end{minipage}\hspace{0.02\textwidth}%
\begin{minipage}{0.42\textwidth}
\centering
\begin{tikzpicture}[
    scale=0.9,
    every node/.style={font=\small},
    singleton/.style={circle, draw, fill=black, inner sep=1.8pt},
    blowup/.style={circle, draw, thick, minimum size=11.5mm, fill=gray!10},
    edge/.style={thick}
]

\node[singleton] (u1) at (90:2.15)  {};
\node[singleton] (u2) at (142:2.15) {};
\node[blowup]    (u3) at (194:2.15) {};
\node[singleton] (u4) at (246:2.15) {};
\node[blowup]    (u5) at (298:2.15) {};
\node[blowup]    (u6) at (350:2.15) {};
\node[blowup]    (u7) at (38:2.15)  {};

\draw[edge] (u1) -- (u2);
\draw[edge] (u2) -- (u3);
\draw[edge] (u3) -- (u4);
\draw[edge] (u4) -- (u5);
\draw[edge] (u5) -- (u6);
\draw[edge] (u6) -- (u7);
\draw[edge] (u7) -- (u1);

\end{tikzpicture}

\vspace{1.5mm}
{\small (b) A graph in $\mathfrak C_7^{\ge 2}\setminus \mathfrak C_7^{2}$}
\end{minipage}

\caption{Two special blowups of $C_7$.}
\label{fig: special blowup of C7}
\end{figure}


\subsection{A stability starting point}

We start from the following stability theorem: if $H$ is homomorphic to an odd cycle of length at least $2s-1$, then every sufficiently dense $H$-free graph is almost bipartite. Recall that $\mathfrak C_s^{\ge 0}$ is simply the family of all blowups of $C_s$.

\begin{lemma}[\cite{LUCZAK20083998}]\label{lem: luczak}
    For any integer $s \ge 2$, let $H \in \mathfrak{C}_{2s-1}^{\ge 0}$. Then for every $\gamma, \eta > 0$, there exists $n_0$ such that for every $H$-free graph $G$ with $|V(G)| = n \ge n_0$ and        $\delta(G) \ge \left( \frac{2}{2s+1} + \eta \right) n,$
    one can remove at most $\gamma n^2$ edges from $G$ to make it bipartite.
\end{lemma}

We shall also need the following simple fact.
For a set \(X\), we write \(N(X)=\bigcap_{x\in X}N(x)\) for its common neighbourhood.

\begin{fact}\label{fact:pigeonhole-principle}
For any $\rho >0$ and positive integer $m$, there exists $n_0 = n_0(\rho, m)$ such that the following holds.
Let $G = (X\cup Y, E)$ be a bipartite graph such that $|X|,|Y|\ge n_0$.
Suppose that for every $x\in X$, we have $|N(x)\cap Y|\ge (\frac{1}{2}+\rho)|Y|$.
Then for any subset $P\subset X$ of size $m$, there exists a subset $Q\subseteq P$ of size at least $(\frac{1}{2}+\frac{\rho}{2})m$ such that there exists $L\subset N(Q)\cap Y$ of size at least $m$.
\end{fact}

\begin{proof}
By double counting, we have
\begin{align*}
\sum_{y\in Y}|N(y)\cap P|=
\sum_{x\in P}|N(x)\cap Y|
\ge (\frac{1}{2}+\frac{\rho}{2})m\cdot |Y|.
\end{align*}
Hence, by the pigeonhole principle, there exists $y_0\in Y$ such that $|N(y_0)\cap P|\ge (\frac{1}{2}+\frac{\rho}{2})m$.
Let $Y_0 = Y\setminus \{y_0\}$.
Since $n_0$ is sufficiently large, we still have $|N(x)\cap Y_0|\ge (\frac{1}{2}+\frac{\rho}{2})|Y_0|$ for every $x\in X$.
Repeating the same argument, there exists $y_1\in Y_0$ such that $|N(y_1)\cap P|\ge (\frac{1}{2}+\frac{\rho}{2})m$.
Let $m_0 = m 2^m$.
Since $n_0$ is sufficiently large, we may repeat the above argument $m_0$ times and obtain a subset $L_0\subseteq Y$ such that, for every $y\in L_0$, we have $|N(y)\cap P|\ge (\frac{1}{2}+\frac{\rho}{2})m$.
Finally, by the pigeonhole principle, there exists a subset $L\subset L_0$ of size at least $|L_0|/2^m \ge m$ such that there exists $Q\subset P$ of size at least $(\frac{1}{2}+\frac{\rho}{2})m$ with $L\subset N(Q)\cap Y$.
This completes the proof.
\end{proof}

\subsection{Proof of the upper bound}
\begin{proof}[Proof of \Cref{thm: uppbdd for 1/2k}]
For any $\eps>0$, let $n$ be an integer sufficiently large compared with $1/\eps$ and $|V(H)|$, and let $G$ be a maximal $H$-free graph on $n$ vertices with $\delta(G)\ge \left(\frac{1}{4}+\eps\right)n$.
It suffices to show that $G$ is bipartite.

By Lemma~\ref{lem: luczak}, for $n$ sufficiently large and $s\ge 4$, we can delete at most $\eps^{5}n^2$ edges from $G$ to make it bipartite.
Let $E_0$ denote the set of deleted edges.
Let 
\begin{align*}
    Z = \{z\in V(G) \mid \text{ there are at least $\eps^2 n$ vertices $v$ such that $zv\in E_0$}\}.
\end{align*}
Then $|Z|\le \eps^2 n$.
Thus, $G$ admits a partition $V(G) = Z\cup A\cup B$ such that $|Z|\le \eps^2 n$ and $\max\{\Delta(G[A]),\Delta(G[B])\}\le \eps^2 n$. Starting from this partition, repeatedly move a vertex $z\in Z$ into $A$ if $|N(z)\cap A|\le \eps^2 n$, and into $B$ if $|N(z)\cap B|\le \eps^2 n$.  The process stops after at most the initial value of $|Z|$ moves.  Relabelling the final partition as $Z\cup A\cup B$, we still have $|Z|\le\eps^2 n$, every vertex left in $Z$ has at least $\eps^2 n$ neighbours in each of $A$ and $B$, and
\[
    \max\{\Delta(G[A]),\Delta(G[B])\}\le 2\eps^2 n.
\]
For any $(a,b)\in A\times B$, we have
$\min\{ |N(a)\cap B|, |N(b)\cap A|\}\ge \left(\frac{1}{4}+ \frac{2\eps}{3}\right)n.$

Set $m_H:=(2s-1)|V(H)|$, and fix an integer $M=M(H,\eps)$ such that
$\eps M/5\ge m_H$. We shall use the following observation. Suppose that one of
the configurations below contains a blowup of $C_{2s-1}$ in which a consecutive
interval of at most four cycle parts is prescribed to be singleton vertices, and
all remaining cycle parts have size at least $|V(H)|$. Then the configuration
contains $H$. Indeed, take a witnessing representation of
$H\in\mathfrak C_{2s-1}^{\ge4}$ and rotate it so that the prescribed singleton
interval is contained in the singleton interval of $H$. The remaining host parts
are large enough to accommodate the corresponding parts of $H$; if the same large
set is used for several cycle parts, split it into disjoint pieces of size
$|V(H)|$.

We first show that $G[A\cup B]$ is bipartite.
By symmetry, it suffices to show that $G[A]$ is empty.
Suppose, for a contradiction, that there exist $x,y\in A$ such that $xy\in E(G)$.
If $|B|\le n/2$, then since $|N(x)\cap B|+ |N(y)\cap B|\ge (\frac{1}{2}+\eps)n$, we have $|N(x)\cap N(y)\cap B|\ge \eps n/2$.
By Fact~\ref{fact:pigeonhole-principle} with $P = N(x)\cap N(y)\cap B$, there exist $P_1\subseteq P$ and $P_2\subseteq N(P_1)\cap A$, both of size $M$, such that
$G[P_1,P_2]$ is complete bipartite. Then $G[\{x,y\}\cup P_1\cup P_2]$ contains a copy of $H$, a contradiction.

If $|B|> n/2$, then by the argument above, we may assume that $|N(x)\cap N(y)\cap B|\le \eps n/2$.
Let $Q_1\subseteq N(x)\cap B$ and $Q_2\subseteq N(y)\cap B$ be disjoint subsets of $B$, both of size $M$; this is possible for $n$ large because each of $N(x)\cap B$ and $N(y)\cap B$ has linear size while their intersection has size at most $\eps n/2$.
Since $|A|<n/2$, every vertex of $B$ has at least $(\frac12+\eps/5)|A|$ neighbours in $A$.
By Fact~\ref{fact:pigeonhole-principle} with $P = Q_1\cup Q_2$ and $\rho = \eps /5$, there exist 
$Q\subset Q_1\cup Q_2$ of size at least $(\frac{1}{2}+\frac{\eps}{10})\cdot 2M = M + \eps M/5$ such that there exists $L\subset A\cap N(Q)$ of size at least $\eps M/5$.
Thus there exists 
$Q_1'\subset Q_1\cap Q$ and $Q_2'\subset Q_2\cap Q$ of size at least $\eps M/5$.
Since $L\subset A\cap N(Q_1'\cup Q_2')$, the graph
$G[\{x,y\}\cup Q_1'\cup L\cup Q_2']$ contains a blowup of $C_{2s-1}$: use
$x,y$ as two consecutive singleton parts, and use disjoint pieces of
$Q_2'$, $L$, and $Q_1'$ for the remaining path from $y$ back to $x$.
Hence it contains $H$, a contradiction.
It remains to prove the following claim.

\begin{claim}\label{claim: Z=empty k=2}
    $Z = \varnothing$.
\end{claim}

\begin{poc}
    Suppose that $Z$ is non-empty.
    By the construction of the partition, every $z\in Z$ satisfies
    \begin{align}\label{ineq: k=2 ratio at least 1/2}
        \min \{|N(z)\cap A|, |N(z)\cap B|\}\ge \eps^2n.
    \end{align}
    By symmetry, assume that $|B|\le n/2$.
    Then for any $a\in A$, we have $\frac{|N(a)\cap B|}{|B|}\ge \frac{1}{2} + \eps$.
    Since $n$ is sufficiently large, Fact~\ref{fact:pigeonhole-principle}, applied with $P = N(z)\cap B$, gives sets $P_1\subseteq N(z)\cap B$ and $P_2\subseteq A$, both of size $M$ such that $G[P_1,P_2]$ forms a copy of $K_{M,M}$.
    If $|P_2\cap (N(z)\cap A)| \ge M/2$, then
    $G[\{z\}\cup P_1\cup (P_2\cap N(z)\cap A)]$ contains a blowup of
    $C_{2s-1}$ with $z$ as a singleton part, and hence contains $H$, a contradiction.
    Thus there exist subsets $P_2'\subseteq P_2$ and $Q\subseteq N(z)\cap A$, both of size $M/2$, such that $P_2'\cap Q = \varnothing$.
    Since for any $a\in P_2'\cup Q\subseteq A$, we have $\frac{|N(a)\cap B|}{|B|}\ge \frac{1}{2} + \eps$,
    by \eqref{ineq: k=2 ratio at least 1/2} and Fact~\ref{fact:pigeonhole-principle}, applied with $P = P_2' \cup Q$, there exists a set $L\subseteq B$ of size $M/2$ and subsets $P_2''\subseteq P_2'$ and $Q'\subseteq Q$, both of size at least $\eps M/4$, such that for every $v\in L$, we have $(P_2''\cup Q')\subseteq N(v)$.
    Then $G[\{z\}\cup P_1\cup P_2''\cup Q'\cup L]$ contains a blowup of
    $C_{2s-1}$ with $z$ as a singleton part: take a path from $P_1$ to $Q'$ of length $2s-3$ through disjoint pieces of $P_2''$ and $L$, and close it through $z$. Hence it contains $H$, a contradiction.
\end{poc}

\begin{center}
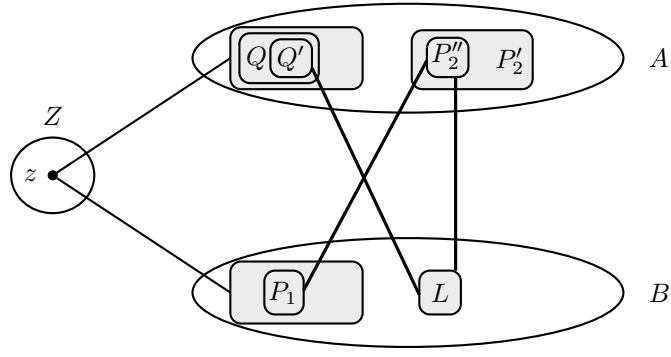

\begin{tikzpicture}[
scale=1,
every node/.style={font=\small},
part/.style={draw, thick},
region/.style={draw, thick, rounded corners, fill=gray!15},
vtx/.style={circle, draw, inner sep=1.2pt, fill=black}
]

\draw[part] (-3.8,0) ellipse (0.55 and 0.50);
\node at (-3.8,0.78) {$Z$};

\draw[part] (0.9,1.55) ellipse (2.85 and 0.72);
\draw[part] (0.9,-1.55) ellipse (2.85 and 0.72);

\node[anchor=west] at (3.95,1.55) {$A$};
\node[anchor=west] at (3.95,-1.55) {$B$};

\node[vtx,label=left:$z$] (z) at (-3.8,0) {};

\draw[region] (-1.45,1.14) rectangle (0.30,1.96);
\draw[region] (-1.45,-1.96) rectangle (0.30,-1.14);

\draw[region] (-1.33,1.22) rectangle (-0.28,1.88);
\node at (-1.10,1.55) {$Q$};

\draw[region] (-0.92,1.30) rectangle (-0.37,1.80);
\node at (-0.65,1.55) {$Q'$};

\draw[region] (-1.00,-1.84) rectangle (-0.48,-1.26);
\node at (-0.755,-1.55) {$P_1$};

\draw[region] (0.95,1.14) rectangle (2.55,1.92);
\node[anchor=south] at (2.25,1.20) {$P_2'$};

\draw[region] (1.15,1.30) rectangle (1.70,1.82);
\node at (1.425,1.59) {$P_2''$};

\draw[region] (1.05,-1.84) rectangle (1.60,-1.26);
\node at (1.325,-1.55) {$L$};

\draw[thick] (z) -- (-1.45,1.55);
\draw[thick] (z) -- (-1.45,-1.55);

\draw[very thick] (1.53,-1.26) -- (1.53,1.30);   
\draw[very thick] (1.05,-1.58) -- (-0.37,1.42);  
\draw[very thick] (-0.48,-1.52) -- (1.15,1.52);  

\end{tikzpicture}

\captionof{figure}{The subgraph $G[\{z\}\cup P_1\cup P_2''\cup Q'\cup L]$ contains a blowup of $C_5$.}
\label{fig:upperbound-one-quarter}

\end{center}

    Therefore, $G$ is bipartite.
    Since $G$ is maximal $H$-free, it follows that $G$ is a complete bipartite graph.
    This completes the proof.
\end{proof}

\section{Concluding remarks}\label{sec:concluding-remarks}

The results of this paper suggest that the blowup-threshold spectrum
$\Delta_{\textup{B}}
    =
    \{\delta_{\textup{B}}(H):\chi(H)\ge 3\}$
is more subtle than the chromatic-threshold spectrum.  The positivity
theorem shows that the blowup threshold never vanishes on non-bipartite graphs,
the non-monotonicity theorem shows that it is sensitive to global features of the
forbidden graph, and the exact value \(\frac{1}{4}\) gives a concrete value
which is not a possible value of \(\delta_\chi(H)\).  Even the \(3\)-chromatic part of the spectrum $\Delta_{\textup{B}}^{(3)}
    =
    \{\delta_{\textup{B}}(H):\chi(H)=3\}$ remains largely open.

\begin{ques}\label{ques:three-chromatic-blowup-spectrum}
Determine the set
$\Delta_{\textup{B}}^{(3)}
    =
    \{\delta_{\textup{B}}(H):\chi(H)=3\}.$
\end{ques}

A particularly natural boundary value is \(1/2\).  It is the largest possible
chromatic threshold for a \(3\)-chromatic forbidden graph, and our
non-monotonicity construction shows that pseudo-blowup constructions at density
\(1/2\) can behave in a delicate way.  It would be useful to understand structural feature of \(H\), for when this value occurs.

\begin{ques}\label{ques:half-blowup-threshold}
Characterize the \(3\)-chromatic graphs \(H\) with
$\delta_{\textup{B}}(H)=\frac{1}{2}.$
\end{ques}

An ambitious possibility is that the whole spectrum may contain
many rational values.

\begin{ques}\label{ques:dense-spectrum}
Is \(\Delta_{\textup{B}}\) dense in some non-degenerate interval \(I\subseteq (0,1]\)? 
\end{ques}

Finally, the failure of monotonicity suggests that graph operations may have a
nontrivial effect on \(\delta_{\textup{B}}\).  It would be interesting to
understand how the blowup threshold behaves under e.g.~disjoint union, joins, adding
isolated vertices, and taking induced subgraphs.  Such structural principles may
be necessary before a systematic classification of blowup thresholds can be
attempted.

\medskip

\noindent\textbf{Acknowledgement.}
The authors thank Zixiang Xu for helpful discussions.

\bibliographystyle{abbrv}
\bibliography{BlowupThreshold}

@misc{ning2026chromatic,
title={On the chromatic profile for tripartite graphs and beyond},
author={Ning, Bo and Wang, Jian and Xue, Yisai},
note={arXiv preprint: 2604.09394},
year={2026}
}

@misc{bottcher2023graphs,
title={Graphs with large minimum degree and no small odd cycles are {$3$}-colourable},
author={B{\"o}ttcher, Julia and Frankl, N{\'o}ra and Cecchelli, Domenico Mergoni and Parczyk, Olaf and Skokan, Jozef},
note={arXiv preprint: 2302.01875},
year={2023}
}

@article{AlonFischerNewman,
title={Efficient testing of bipartite graphs for forbidden induced subgraphs},
author={Alon, Noga and Fischer, Eldar and Newman, Ilan},
journal={SIAM Journal on Computing},
volume={37},
number={3},
pages={959--976},
year={2007},
publisher={SIAM}
}

@incollection{LovaszSzegedy,
title={Regularity partitions and the topology of graphons},
author={Lov{\'a}sz, L{\'a}szl{\'o} and Szegedy, Bal{\'a}zs},
booktitle={An Irregular Mind: Szemer{\'e}di is 70},
pages={415--446},
year={2010},
publisher={Springer}
}

@misc{huang2025interpolatingchromatichomomorphismthresholds,
title={Interpolating chromatic and homomorphism thresholds},
author={Xinqi Huang and Hong Liu and Mingyuan Rong and Zixiang Xu},
note={arXiv preprint: 2502.09576},
year={2025},
}

@article{LUCZAK20083998,
title = {On the minimum degree forcing {$F$}-free graphs to be (nearly) bipartite},
journal = {Discrete Mathematics},
volume = {308},
number = {17},
pages = {3998-4002},
year = {2008},
issn = {0012-365X},
doi = {https://doi.org/10.1016/j.disc.2007.06.047},
url = {https://www.sciencedirect.com/science/article/pii/S0012365X07005675},
author = {Tomasz Łuczak and Miklós Simonovits},
}

@misc{2024GraphToGeom,
title={Beyond the chromatic threshold via {$(p,q)$}-theorem, and a sharp blow-up phenomenon},
author = {Liu, H. and Shangguan, Ch. and Skokan, J. and Xu, Z.},
note={arXiv preprint: 2403.17910},
year={2024},
}

@misc{2010arxivKrfree,
title={Chromatic number and minimum degree of {$K_r$}-free graphs},
author={Nikiforov, V.},
note={arXiv preprint: 1001.2070},
year={2010},
}

@misc{2022Maya,
title={Homotopy and the Homomorphism Threshold of Odd Cycles},
author={M. Sankar},
note={arXiv preprint: 2206.07525},
year={2022},
}

@article {2013advAllChromatic,
AUTHOR = {Allen, P. and B{\"o}ttcher, J. and Griffiths, S. and
Kohayakawa, Y. and Morris, R.},
TITLE = {The chromatic thresholds of graphs},
JOURNAL = {Adv. Math.},
FJOURNAL = {Advances in Mathematics},
VOLUME = {235},
YEAR = {2013},
PAGES = {261--295},
ISSN = {0001-8708,1090-2082},
MRREVIEWER = {Derrick\ Paul\ Stolee},
DOI = {10.1016/j.aim.2012.11.016},
URL = {https://doi.org/10.1016/j.aim.2012.11.016},
}

@article {2011JGTKrChromatic,
AUTHOR = {Goddard, W. and Lyle, J.},
TITLE = {Dense graphs with small clique number},
JOURNAL = {J. Graph Theory},
FJOURNAL = {Journal of Graph Theory},
VOLUME = {66},
YEAR = {2011},
NUMBER = {4},
PAGES = {319--331},
ISSN = {0364-9024,1097-0118},
MRREVIEWER = {Myriam\ Preissmann},
DOI = {10.1002/jgt.20505},
URL = {https://doi.org/10.1002/jgt.20505},
}

@article {2002Thomassen,
AUTHOR = {Thomassen, C.},
TITLE = {On the chromatic number of triangle-free graphs of large
minimum degree},
JOURNAL = {Combinatorica},
FJOURNAL = {Combinatorica. An International Journal on Combinatorics and
the Theory of Computing},
VOLUME = {22},
YEAR = {2002},
NUMBER = {4},
PAGES = {591--596},
ISSN = {0209-9683,1439-6912},
MRREVIEWER = {Mirko\ Hor\v{n}'{a}k},
DOI = {10.1007/s00493-002-0009-5},
URL = {https://doi.org/10.1007/s00493-002-0009-5},
}

@article {2020COMBHomoOddCycle,
AUTHOR = {Ebsen, Oliver and Schacht, Mathias},
TITLE = {Homomorphism thresholds for odd cycles},
JOURNAL = {Combinatorica},
FJOURNAL = {Combinatorica. An International Journal on Combinatorics and
the Theory of Computing},
VOLUME = {40},
YEAR = {2020},
NUMBER = {1},
PAGES = {39--62},
ISSN = {0209-9683,1439-6912},
MRCLASS = {05C35 (05C07 05C15 05D40)},
MRNUMBER = {4078811},
DOI = {10.1007/s00493-019-3920-8},
URL = {https://doi.org/10.1007/s00493-019-3920-8},
}

@article {1973ErdosSimonovits,
AUTHOR = {Erd\H{o}s, P. and Simonovits, M.},
TITLE = {On a valence problem in extremal graph theory},
JOURNAL = {Discrete Math.},
FJOURNAL = {Discrete Mathematics},
VOLUME = {5},
YEAR = {1973},
PAGES = {323--334},
ISSN = {0012-365X,1872-681X},
MRREVIEWER = {D.\ R.\ Lick},
DOI = {10.1016/0012-365X(73)90126-X},
URL = {https://doi.org/10.1016/0012-365X(73)90126-X},
}

@misc{2010ColoringViaVCDim,
title={Coloring dense graphs via {VC}-dimension},
author={{\L}uczak, T. and Thomass{\'e}, S.},
note={arXiv preprint: 1007.1670},
year={2010},
}

@unpublished{2011Unpubilished,
title={Dense triangle-free graphs are four-colorable: A solution to the {E}rd{\H{o}}s-{S}imonovits problem},
author={Brandt, S. and Thomass{\'e}, S.},
note={preprint},
year={2011}
}

@Article{2020CPCProb,
Author = {Oberkampf, H. and Schacht, M.},
Title = {On the structure of dense graphs with bounded clique number},
FJournal = {Combinatorics, Probability and Computing},
Journal = {Comb. Probab. Comput.},
ISSN = {0963-5483},
Volume = {29},
Number = {5},
Pages = {641--649},
Year = {2020},
}

@article {2006CombTriangleHom,
AUTHOR = {{\L}uczak, T.},
TITLE = {On the structure of triangle-free graphs of large minimum
degree},
JOURNAL = {Combinatorica},
FJOURNAL = {Combinatorica. An International Journal on Combinatorics and
the Theory of Computing},
VOLUME = {26},
YEAR = {2006},
NUMBER = {4},
PAGES = {489--493},
ISSN = {0209-9683,1439-6912},
DOI = {10.1007/s00493-006-0028-8},
URL = {https://doi.org/10.1007/s00493-006-0028-8},
}

@article {1995DMJin,
AUTHOR = {Jin, G. P.},
TITLE = {Triangle-free four-chromatic graphs},
JOURNAL = {Discrete Math.},
FJOURNAL = {Discrete Mathematics},
VOLUME = {145},
YEAR = {1995},
NUMBER = {1-3},
PAGES = {151--170},
ISSN = {0012-365X,1872-681X},
MRREVIEWER = {David\ E.\ Woolbright},
DOI = {10.1016/0012-365X(94)00063-O},
URL = {https://doi.org/10.1016/0012-365X(94)00063-O},
}

@incollection {1978OriginalRegularity,
AUTHOR = {Szemer{\'e}di, E.},
TITLE = {Regular partitions of graphs},
BOOKTITLE = {Probl{\`e}mes combinatoires et th{\'e}orie des graphes
({C}olloq. {I}nternat. {CNRS}, {U}niv. {O}rsay, {O}rsay,
1976)},
SERIES = {Colloq. Internat. CNRS},
VOLUME = {260},
PAGES = {399--401},
PUBLISHER = {CNRS, Paris},
YEAR = {1978},
ISBN = {2-222-02070-0},
MRREVIEWER = {D.\ A.\ Klarner},
}

@incollection {1982OddCycles,
AUTHOR = {H{\"a}ggkvist, R.},
TITLE = {Odd cycles of specified length in nonbipartite graphs.},
BOOKTITLE = {Graph theory ({C}ambridge, 1981)},
SERIES = {North-Holland Math. Stud., 62},
PAGES = {89--99},
PUBLISHER = {},
YEAR = {1982},
}

@article {2007OddCycleChromatic,
AUTHOR = {Thomassen, C.},
TITLE = {On the chromatic number of pentagon-free graphs of large
minimum degree},
JOURNAL = {Combinatorica},
FJOURNAL = {Combinatorica. An International Journal on Combinatorics and
the Theory of Computing},
VOLUME = {27},
YEAR = {2007},
NUMBER = {2},
PAGES = {241--243},
ISSN = {0209-9683,1439-6912},
DOI = {10.1007/s00493-007-0054-1},
URL = {https://doi.org/10.1007/s00493-007-0054-1},
}

@article {2019JGTC3C5,
AUTHOR = {Letzter, S. and Snyder, R.},
TITLE = {The homomorphism threshold of {${C_3,C_5}$}-free graphs},
JOURNAL = {J. Graph Theory},
FJOURNAL = {Journal of Graph Theory},
VOLUME = {90},
YEAR = {2019},
NUMBER = {1},
PAGES = {83--106},
ISSN = {0364-9024,1097-0118},
MRREVIEWER = {Andr\'{e}\ E.\ K\'{e}zdy},
DOI = {10.1002/jgt.22369},
URL = {https://doi.org/10.1002/jgt.22369},
}
\end{document}